\documentclass[a4paper,10pt]{article}
\usepackage[margin=1in]{geometry}

\usepackage[final]{microtype}
\usepackage{amsthm,amsmath,amssymb}
\makeatletter
        \newcommand{\blx@nowarnpolyglossia}{}
\makeatother

\usepackage{csquotes}

\usepackage{graphicx}
\usepackage{xcolor}

\usepackage{float}
\floatstyle{plain}
\restylefloat{figure}

\usepackage{subcaption}
\usepackage{multirow}

\usepackage{pgfplots}
\pgfplotsset{compat=newest}
\usepgfplotslibrary{fillbetween}
\usepackage{tikz}
\usetikzlibrary{graphs.standard,intersections,patterns,calc}

\usepackage[backend=biber, bibstyle=mynumeric, citestyle=numeric, sorting=anyt, autopunct=true,
maxcitenames=5, maxbibnames=10, natbib=true, isbn=true, arxiv=abs, doi=true, url=false,
eprint=true, giveninits=true]{biblatex}

\usepackage{hyperref}
\hypersetup{%
        unicode=true,
        pdftitle={Approximate Sampling of Graphs with Near-P-stable Degree
        Intervals},
        pdfauthor={Péter L. Erdős, Tamás Róbert Mezei, István Miklós},
        pdfsubject={Markov chains, network science, graph theory},
        pdfkeywords={degree sequences, realizations, switch Markov chain,
            rapidly mixing, Sinclair's multi-commodity flow method, P-stability,
            weak P-stability},
	    hidelinks,
        colorlinks=false,       %
}

\usepackage{enumitem}
\usepackage{cleveref}

\usepackage{dsfont}

\newtheorem{theorem}{Theorem}[section]

\newtheorem{lemma}[theorem]{Lemma}

\theoremstyle{definition}
\newtheorem{definition}[theorem]{Definition}

\newtheorem{remark}[theorem]{Remark}

\DeclareMathOperator{\dom}{dom}
\newcommand\DOI[1]{\href{https://doi.org/#1}{\tt DOI:#1}}

\usepackage{authblk}

\title{\Large\bfseries Approximate Sampling of Graphs with
Near-$P$-stable Degree Intervals\footnote{PLE, TRM, and IM were supported in part by the National Research, Development and
Innovation Office --- NKFIH grants SNN 135643 and K 132696.}}
\author[1]{Péter L.\ Erdős}
\author[1]{Tamás Róbert Mezei}
\author[1,2]{István Miklós}
\affil[1]{\small Alfréd Rényi Institute of Mathematics (LERN)\protect\\ Reáltanoda utca 13--15, H-1053 Budapest,
Hungary.\protect\\\texttt{<erdos.peter,
mezei.tamas.robert,miklos.istvan>@renyi.hu}\protect\\\ }
\affil[2]{\small Institute for Computer Science and Control, SZTAKI (LERN)\protect\\ Kende utca 13--17, H-1111 Budapest,
Hungary.}

\begin{document}
\maketitle

\begin{abstract}
    The approximate uniform sampling of graph realizations with a given degree
    sequence is an everyday task in several social science, computer science,
    engineering etc.\ projects. One approach is using Markov chains. The best
    available current result about the well-studied switch Markov chain is that
    it is rapidly mixing on $P$-stable degree sequences
    (see~\DOI{10.1016/j.ejc.2021.103421}). The switch Markov chain does not
    change any degree sequence. However, there are cases where degree intervals
    are specified rather than a single degree sequence. (A natural scenario
    where this problem arises is in hypothesis testing on social networks that
    are only partially observed.) \citeauthor{rechner_uniform_2018} introduced
    in \citeyear{rechner_uniform_2018} the notion of \emph{degree interval
    Markov chain} which uses three (separately well-studied) local
    operations (switch, hinge-flip and toggle), and employing on degree
    sequence realizations where any two sequences under scrutiny have very
    small coordinate-wise distance. Recently
    \citeauthor{amanatidis_approximate_2021} published a beautiful paper
    (\href{https://arxiv.org/abs/2110.09068}{\texttt{arXiv:2110.09068}}),
    showing that the degree interval Markov chain is rapidly mixing if the
    sequences are coming from a system of very thin intervals which are
    centered not far from a regular degree sequence. In this paper we extend
    substantially their result, showing that the degree interval Markov
    chain is rapidly mixing if the intervals are centred at $P$-stable
    degree sequences.

    \bigskip

    \noindent\textbf{Keywords}
		{
			\small
			degree sequences,
			realizations,
			switch Markov chain,
			rapidly mixing,
			Sinclair's multi-commodity flow method,
			P-stability, weak P-stability
		}
\end{abstract}

\section{Introduction}\label{sec:intro}

In this relatively short, highly technical paper we prove a substantial extension of a
recent result of \citet[Theorem 1.3]{amanatidis_approximate_2021}. Our proof is
based on the unified approach that was developed in~\cite{erdos_mixing_2022} for
$P$-stable degree sequences. For sake of brevity in this section we concisely describe
the problem itself, but we will not give a detailed description of the
background. For further details, the diligent reader is referred
to~\cite{amanatidis_approximate_2021,erdos_mixing_2022}.

\medskip

Approximate sampling graphs with given degree sequences play increasingly
important role in modelling different real-life dynamics. One basic way to study
them is the switch Markov chain method, made popular by
\citet{kannan_simple_1999}. The currently best result via this method
is~\cite{erdos_mixing_2022} where it is proved that the switch Markov chain is
rapidly mixing on $P$-stable degree sequences. The notion of $P$-stability was
introduced by \citet{jerrum_fast_1990} and studied for its own sake at first by
\citet{jerrum_when_1992}.

\medskip

In real-life applications it is not always possible to know the exact degree
sequence of the targeted network. For example a natural scenario where this
problem arises is in hypothesis testing on social networks that are only
partially observed. Therefore it can happen that we have to sample networks with
slightly different degree sequences. It is possible to study the situation via
Markov chain decompositions, where there is another Markov chain to move among
the component chains. A good example for this approach is the proof
of~\cite[Theorem 1.1]{amanatidis_approximate_2021}.

\medskip

Another possibility is to introduce further local operations, since the switch
operation itself does not change the degree sequence. Such operations are the
\emph{hinge flip} and the \emph{toggle} (the deletion-insertion) operations.
These two latter operations was introduced by Jerrum and Sinclair in their
seminal work about approximate 0{-}1
permanents~\cite{jerrum_approximating_1989}. (The number of perfect matchings of a bipartite
graph is equal to the permanent of the bipartite adjacency matrix.)
These three operations together are often applied in network building
applications in \emph{practice} (as it was pointed out
in~\cite{coolen_generating_2017}) but without any theoretical insurance for the
correct result.

\medskip

In~\citeyear{rechner_uniform_2018} \citet{rechner_uniform_2018} defined a Markov
chain with these three local operations for bipartite graphs. Amanatidis and
Kleer recognized in their important, recent
preprint~\cite{amanatidis_approximate_2021} the following very interesting fact:
assume that the inconsistencies in the degree sequences are never bigger than
one (the degrees can be $i$ or $i+1$) coordinate-wise, and the degree
intervals are placed close to a given constant $r$ (the interval placements can
vary between $[r-r^\alpha, r+r^\alpha]$ where alpha is at most $1/2$). The
authors coined the name \emph{near-regular degree intervals} for this degree
sequence property and the name  \emph{degree interval Markov chain} for this
whole setup. Their result is that the degree interval
Markov chain for near-regular degree intervals is rapidly mixing.

\medskip

Our main result (\Cref{thm:main}) is that this Markov chain is rapidly mixing
for such tight degree intervals where they are placed at $P$-stable degree
sequences. Since all degree sequences close to some constant are $P$-stable, but
$P$-stable degree sequences can be very far from regular sequences, our result
is clearly a very extensive generalization of the theorem of Amanatidis and
Kleer.

\medskip

To our great surprise, it turned out that this result can be derived from the
proof of the main theorem of~\cite{erdos_mixing_2022}. For that end we had to
analyse in detail the auxiliary structures of the proof and to extend to cover
this setup. The result of this analysis is the notion of \emph{precursor}
(\Cref{sec:precur}). In turn this notion is conducive to a rather short
proof of the rapidly mixing property. Therefore the main task in this paper is
to define the appropriate precursor.

\section{Definitions and Notation}\label{sec:def}

Many of the definitions in this section are extensions or generalizations of
notions introduced in~\cite{erdos_mixing_2022}. We will alert the reader
whenever this is the case.

\medskip

We consider $\mathbb{N}$ the set of non-negative integers. Let
$[n]=\{1,\ldots,n\}$ denote the integers from 1 to $n$, and let $\binom{[n]}{k}$
denote the set of $k$-element subsets of $[n]$. Given a subset $S\subseteq [n]$,
let $\mathds{1}_S: [n]\to \{0,1\}$ be the characteristic function of $S$, that
is, $\mathds{1}_S(s)=1\Leftrightarrow s\in S$. We often use $\uplus$ to
emphasize that a union of pairwise disjoint sets is taken. The graphs in this
paper are vertex-labelled and finite. Parallel edges and loops are forbidden,
and unless otherwise stated, the labelled vertex set of an $n$-vertex graph is
$[n]$. The line graph $L(G)$ of a graph $G$ is a graph on the vertex set $E(G)$
(so the vertices of $L(G)$ are taken from $\binom{[n]}{2}$), where any two edges
$e,f\in E(G)$ that are adjacent are joined (by an edge). The line graph is also
free of parallel edges and loops. A trail is a walk that does not visit any edge
twice. An open trail starts and ends on two distinct vertices. A closed trail
does not have a start nor an end vertex. Given a matrix $M\in
\mathbb{Z}^{n\times n}$, its $\ell_1$-norm is $\|M\|_1=\sum_{ij}|M_{ij}|$.

\begin{definition}
Given two graphs on $[n]$ as vertices, say, $X=([n],E(X))$ and $Y=([n],E(Y))$,
we define their symmetric difference graph
\begin{equation*}
	X\triangle Y=([n],E(X)\triangle E(Y)).
\end{equation*}
\end{definition}

\begin{definition}
Given a set of edges $R\subseteq \binom{[n]}{2}$, we may treat $R$ as a graph.
If $X$ is a graph on the vertex set $[n]$, let
\begin{equation*}
	X\triangle R=([n],E(X)\triangle R).
\end{equation*}
\end{definition}

\begin{definition}
    A \emph{degree sequence} on $n$ vertices is a vector $d\in \mathbb{N}^n$
    which is coordinate-wise at most $n-1$. The set
    of \emph{realizations} of $d$ denotes the following set of graphs:
    \begin{equation*}
        \mathcal{G}(d)=\left\{ G\ \Big|\ V(G)=[n],\ \deg_G(i)=d_i\ \forall
        i\in[n]\right\}.
    \end{equation*}
    The degree sequence $d$ is \emph{graphic} if $\mathcal{G}(d)$ is non-empty.
    A set of degree sequences $\mathcal{D}$ may contain graphic as well as
    non-graphic degree sequences.
\end{definition}

\begin{definition}
For a pair of vectors $\ell,u\in \mathbb{N}^n$ we write $\ell\le u$ if and only
if $\ell$ is coordinate-wise less than or equal to $u$, that is,
$\ell_i\le u_i$ for all $i\in [n]$. Furthermore, let
\begin{equation*}
    [\ell,u]=\left\{ d\in \mathbb{N}^n\ |\ \ell\le d\le u\right\}.
\end{equation*}
\end{definition}

\begin{definition}
    If $\ell\le u$ are both degree sequences of length $n$, then $[\ell,u]$ is a
    \emph{degree sequence interval}.
    A degree sequence interval $[\ell,u]$ is called \emph{thin} if $u_i\le
    \ell_i+1$ for all $i\in [n]$. We denote the set of \emph{realizations} of the degree
    sequence interval $[\ell,u]$ by
    \begin{equation*}
	    \mathcal{G}(\ell,u)=\bigcup_{d\in [\ell,u]}\mathcal{G}(d).
    \end{equation*}
\end{definition}
\begin{remark}
    Not every degree sequence in $[\ell,u]$ is necessarily graphic, even if both
    $\ell$ and $u$ are graphic.
\end{remark}

\begin{definition}\label{def:pstable}
    Given a polynomial $p\in\mathbb{R}[x]$, we say that a degree sequence $d\in \mathbb{N}^n$ is
    \emph{$p$-stable} if
    \begin{equation*}
        \Big|\mathcal{G}(d)\cup\bigcup_{\{i,j\}\in
	\binom{[n]}{2}}\mathcal{G}(d+\mathds{1}_{\{i,j\}})\Big|\le p(n)\cdot
        |\mathcal{G}(d)|.
    \end{equation*}
\end{definition}

\begin{definition}
    A set of degree sequences $\mathcal{D}$ is $p$-stable if every
    degree sequence $d\in \mathcal{D}$ is $p$-stable.
\end{definition}
\begin{definition}
    A set of degree sequences $\mathcal{D}$ is $P$-stable if there exists $p\in
    \mathbb{R}[x]$ such that $\mathcal{D}$ is $p$-stable.
\end{definition}
In~\cite{erdos_mixing_2022}, only $P$-stability is defined, but in this paper it is more
convenient to also define $p$-stability.
\begin{remark}
A finite set of degree sequences $\mathcal{D}$ is always $P$-stable.
\end{remark}

Let us introduce a weaker stability notion for degree sequence intervals.
\begin{definition}\label{def:wpstable}
    Given $p\in\mathbb{R}[x]$, we say that a degree sequence interval
    $[\ell,u]\subseteq \mathbb{N}^n$ is \emph{weakly $p$-stable} if
    \begin{equation}\label{eq:wpstable}
        \Big|\bigcup_{\{i,j\}\in
        \binom{[n]}{2}}\mathcal{G}(\ell,u+\mathds{1}_{\{i,j\}})\Big|\le p(n)\cdot
        |\mathcal{G}(\ell,u)|.
    \end{equation}
\end{definition}
\begin{definition}
    A set $\mathcal{I}$ of degree sequence intervals is weakly $P$-stable if there exists $p\in
    \mathbb{R}[x]$ such that every $[\ell,u]\in\mathcal{I}$ is weakly
    $p$-stable. (Any finite $\mathcal{I}$ is weakly $P$-stable.)
\end{definition}
\begin{remark}\label{rem:pstable}
    If the set of degree sequences $[\ell,u]$ is $p$-stable, then $[\ell,u]$ is weakly $p$-stable.
\end{remark}
\begin{remark}\label{rem:notpstable}
    It is possible indeed that $[\ell,u]$ is
    weakly $p$-stable, but $[\ell,u]$ (as a set of degree sequences) is not
    $p$-stable. For example, take $\ell=(0)_{i=1}^n$ and $u={(n-1)}_{i=1}^n$: the
    interval $(\ell,u)$ is clearly $1$-stable, but most of the degree
    sequences on $n$ vertices are not $1$-stable.
\end{remark}

\begin{definition}[Degree interval Markov chain]
    Let us define the degree interval Markov chain $\mathbb{G}(\ell,u)$. The state
    space of the Markov-chain is $\mathcal{G}(\ell,u)$.  In the following we
    define three types of transitions: \emph{switches}, \emph{hinge-flips}, and
    \emph{edge-toggles}. If the current state of the Markov chain is
    $G\in\mathcal{G}(\ell,u)$, then
    \begin{itemize}
	    \item with probability $1/2$, the chain stays in $G$ (the Markov chain
            is lazy),
        \item with probability $1/6$, pick 4 vertices $a,b,c,d$ (uniformly and
            randomly), and the Markov chain changes its state to $G'=G\triangle
            \{ab,cd,ac,bd\}$ if $\deg_{G'}=\deg_G$ (then this is a \emph{switch}),
            otherwise the chain stays in $G$,
        \item with probability $1/6$, pick 3 vertices $a,b,c$ (uniformly and
            randomly), and the Markov chain changes its state to $G''=G\triangle
            \{ab,bc\}$ if $e(G'')=e(G)$ and $G''\in\mathcal{G}(\ell,u)$ (a
            \emph{hinge-flip}), otherwise the chain stays in $G$,
        \item with probability $1/6$, pick a pair of vertices $a,b$ (uniformly
            and randomly), and the Markov chain changes its state to $G'''=G\triangle
            \{ab\}$ if $G'''\in\mathcal{G}(\ell,u)$ (an \emph{edge-toggle}),
            otherwise the chain stays in $G$.
    \end{itemize}
\end{definition}

\begin{figure}[H]
     \centering
     \begin{subfigure}[b]{0.32\textwidth}
         \centering
	 \begin{tikzpicture}
		 \begin{scope}[local bounding box=G1]
			\fill[black!15, very thin] (0.5,-0.5) circle (1.1cm);
			\graph [math nodes,nodes={draw, circle, inner sep=0pt, minimum size=14pt}, grid placement]
			{ subgraph I_n [V={a,b,c,d}, wrap after=2],
				a--b, c--d, a--[densely dashed] c, b--[densely dashed] d};
		\end{scope}
		\begin{scope}[shift={(2.7,0)},local bounding box=G2]
			\fill[black!15, very thin] (0.5,-0.5) circle (1.1cm);
			\graph [math nodes,nodes={draw, circle, inner sep=0pt, minimum size=14pt}, grid placement] {
			subgraph I_n [V={a,b,c,d},wrap after=2], a--[densely dashed]b,
		c--[densely dashed]d, a--c, b--d };
		\end{scope}
		\draw[<->] (G1)--(G2);
	 \end{tikzpicture}
         \caption{Switch}
	 \label{fig:switch}
     \end{subfigure}
     \hfill
     \begin{subfigure}[b]{0.32\textwidth}
         \centering
	 \begin{tikzpicture}
		 \begin{scope}[local bounding box=G1]
			\fill[black!15, very thin] (0.5,-0.5) circle (1.1cm);
			\graph [math nodes,nodes={draw, circle, inner sep=0pt, minimum size=14pt}]
			{ b--a, b--[densely dashed] c };
		\end{scope}
		\begin{scope}[shift={(2.7,0)},local bounding box=G2]
			\fill[black!15, very thin] (0.5,-0.5) circle (1.1cm);
			\graph [math nodes,nodes={draw, circle, inner sep=0pt, minimum size=14pt}]
			{ b--[densely dashed] a, b--c };
		\end{scope}
		\draw[<->] (G1)--(G2);
	 \end{tikzpicture}
         \caption{Hinge-flip}
	 \label{fig:hingeflip}
     \end{subfigure}
     \hfill
     \begin{subfigure}[b]{0.32\textwidth}
         \centering
	 \begin{tikzpicture}
		 \begin{scope}[local bounding box=G1]
			\fill[black!15, very thin] (0.0,-0.5) circle (1.1cm);
			\graph [math nodes,nodes={draw, circle, inner sep=0pt, minimum
			size=14pt}, grid placement] { a--b };
		\end{scope}
		\begin{scope}[shift={(2.7,0)},local bounding box=G2]
			\fill[black!15, very thin] (0.0,-0.5) circle (1.1cm);
			\graph [math nodes,nodes={draw, circle, inner sep=0pt, minimum
			size=14pt}, grid placement] { a--[densely dashed] b };
		\end{scope}
		\draw[<->] (G1)--(G2);
	 \end{tikzpicture}
         \caption{Edge-toggle}
	 \label{fig:edgetoggle}
     \end{subfigure}
     \caption{The three types of operations employed by the degree interval Markov chain.
	Solid~(\raisebox{0.3em}{\protect\tikz{\protect\draw(0,0) -- (0.5,0);}}) and
	dashed~(\raisebox{0.3em}{\protect\tikz{\protect\draw[densely dashed](0,0) -- (0.5,0);}})
	line segments represent edges and non-edges,
	respectively.}\label{fig:operations}
\end{figure}
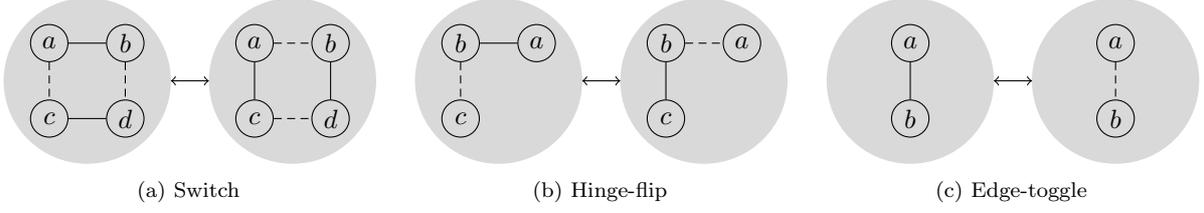

\medskip

We will use the following seminal result of \citeauthor{sinclair_improved_1992}.
Let $\Pr_\mathbb{G}(x\to y)$ denote the transition probability from state $x$ to
$y$ in the Markov chain $\mathbb{G}$.
\begin{theorem}[adapted from {\citet[Proposition 1 and
    Corollary~6']{sinclair_improved_1992}}]\label{thm:sinclair}
	Let $\mathbb{G}$ be an irreducible, symmetric, reversible,
	and lazy Markov chain. Let $f$ be a multicommodity-flow on $\mathbb{G}$
    which sends $\sigma(X)\sigma(Y)$ commodity between any ordered pair $X,Y\in V(\mathbb{G})$,
    where $\sigma\equiv |V(\mathbb{G})|^{-1}$ is the unique stationary distribution on $\mathbb{G}$.
	Then the mixing time of the Markov chain in which it converges
	$\varepsilon$ close in $\ell_1$-norm to $\sigma$ started from any $V(\mathbb{G})$ is
	\begin{equation}\label{eq:sinclair}
		\tau_\mathbb{G}(\varepsilon)\le \rho(f)\cdot\ell(f)\cdot
		\left( \log |\mathbb{G}|-\log \varepsilon\right),
	\end{equation}
	where $\ell(f)$ is the length of the longest path with positive flow and
    $\rho(f)$ is the maximum loading through an oriented edge of the Markov graph
    \begin{equation}\label{eq:rho}
        \rho(f)=\max_{xy\in E(\mathbb{G})}\frac{1}{\sigma(x)\Pr_\mathbb{G}(x\to
        y)}\sum_{xy\in \gamma\in
        \Gamma_{X,Y}}f(\gamma),
    \end{equation}
    where $\Gamma_{X,Y}$ is the set of all simple directed paths from $X$ to
    $Y$ in $\mathbb{G}$.
\end{theorem}
One of the most famous applications of this idea is the result of \citet{jerrum_approximating_1989} providing a probabilistic approximation of the permanent. The following result also relies on \Cref{thm:sinclair}, and it
describes the largest known class of degree sequences where the switch Markov
chain is rapidly mixing (that is, the rate of convergence of the Markov chain is
bounded by a polynomial of the length of the degree sequence).
\begin{theorem}[\cite{erdos_mixing_2022}]\label{thm:switch}
	The switch Markov chain is rapidly mixing on the realizations of any degree
    sequence in a set of $P$-stable degree sequences (the rate of convergence
    depends on the set).
\end{theorem}
There are several known $P$-stable regions, one of the earliest and most
well-known ones is the following.
\begin{theorem}[\citet{jerrum_when_1992}]\label{thm:JMS}
	The set of degree sequences $d$ satisfying
	\begin{equation}\label{eq:JMS}
		{(\Delta-\delta+1)}^2\le 4\delta(n-\Delta+1),\ d\in \mathbb{N}^n
	\end{equation}
	for any $n$ are $P$-stable. (See \Cref{fig:JMS}.)
\end{theorem}

\citet{amanatidis_approximate_2021} recently published a surprising new type of
result, a clever \emph{approximate uniform sampler} (see, for
e.g.~\cite{jerrum_approximating_1989}) for $\mathcal{G}(\ell,u)$ where elements
of $[\ell,u]$ are near regular. They achieve this using a composite
Markov-chain. They also provide the first step in the direction of sampling
$\mathcal{G}(\ell,u)$ directly using the degree interval Markov chain.

\medskip

Let us reiterate that \citet{amanatidis_approximate_2021} apply the Markov chain
suggested by~\citet{rechner_uniform_2018}, which is routinely used in practice.
\begin{theorem}[Theorem~1.3 in~\cite{amanatidis_approximate_2021}]\label{thm:ak21}
    Let $0<\alpha<\frac12$ and $0<\rho<1$ be fixed. Let $r=r(n)$ with $2\le r\le
    (1-\rho)n$. If $[\ell_i,u_i]\subseteq [r-r^\alpha,r+r^\alpha]$ and $u_i-1\le
    \ell_i\le u_i$ for all $i\in [n]$, then the degree interval Markov chain
    $\mathbb{G}(\ell,u)$ is rapidly mixing.
\end{theorem}
Let $w_m$ be the number of realizations in $\mathcal{G}(\ell,u)$ with $m$ edges.
The conditions $u_i-1\le \ell_i\le u_i$ for all $i\in [n]$ are sufficient to
prove that $w_m$ is log-concave, i.e., $w_{m-1}w_{m+1}\le w_m^2$,
see~\cite[Theorem~5.4]{amanatidis_approximate_2021}. The main idea for that
proof is a symmetric-difference decomposition, which we also characterize in our
\textbf{key decomposition lemma}, \Cref{lemma:decomp}.

\paragraph{Our contribution.}
The main objective of this paper is to prove the following theorem.
\begin{theorem}\label{thm:main}
    Suppose $\mathcal{I}$ is a set of weakly $P$-stable and thin degree sequence
    intervals.
    Then the degree interval Markov chain $\mathbb{G}(\ell,u)$ is rapidly mixing
    for any $(\ell,u)\in\mathcal{I}$.
\end{theorem}
It is not hard to see that \Cref{thm:ak21} is a special case of \Cref{thm:main}:
substituting into \cref{eq:JMS}, we get
\begin{equation*}
	{(2r^\alpha+1)}^2 \le 4(r-r^\alpha)(n-r-r^\alpha+1),
\end{equation*}
which holds for any $r$ and $\alpha$ if $n$ is large enough; see \Cref{fig:JMS}.

\medskip

The switch Markov chain can be embedded into the degree interval Markov chain
(the transition probabilities differ by constant factors). Actually, we will use
the proof of \Cref{thm:switch} as a plug-in in the proof of \Cref{thm:main}, so
this paper does not provide a new proof for the switch Markov chain. We will not
consider bipartite and direct degree sequences in this paper, but note that
\Cref{thm:switch} applies to those as well. It is easy to check that the proof
of \Cref{thm:main} works verbatim for bipartite graphs, because the edge-toggles
and hinge-flips are applied on vertices that are joined by paths of odd length
(hence in different classes). In all likelihood, the proof of \Cref{thm:main}
can be probably extended to directed graphs, because directed graphs can be
represented as bipartite graphs endowed with a forbidden 1-factor.

\begin{figure}[H]
	\centering
	\begin{tikzpicture}[scale=0.7]
	\begin{axis}[axis lines=center,
		    ylabel=degree,
	            xlabel=$\frac{\delta+\Delta}{2}$,
	            ytick={0,0.25,0.75,1},
	            yticklabels={$0$,$\frac14n$,$\frac34n$,$n$},
	            xtick={0,0.5,0.8,1},
    		    xticklabels={$0$,{\small $\frac12 n$},{\small $(1-\rho)n$},$n$}]

	\draw[thin,dashed,gray] (0,0.75)-- (0.5,0.75)--(0.5,1);
	\draw[thin,dashed,gray] (0,0.25)-- (0.5,0.25)--(0.5,0);

	\addplot[name path=F,black,domain={0:1},samples=1000] {x^2}
		node[pos=.85, below]{$\delta$};
	\addplot[name path=G,black,domain={0:1},samples=1000] {1-(1-x)^2}
		node[pos=.75, above]{$\Delta$};

	\draw[thin,dashed,gray] (0.8,0)-- (0.8,0.8);

	\addplot[pattern=vertical lines,
		pattern color=black!50]fill
		between[of=F and G, soft clip={domain=0:1}];

	\addplot[name path=H,black!50,domain={0:0.8},samples=1000] {x+0.02*sqrt(x)}
		node[rotate=45,pos=.5, below]{};
	\addplot[name path=I,black!50,domain={0:0.8},samples=1000]
		{x-0.02*sqrt(x)}
		node[pos=.75, above]{};
	\addplot[black!50]fill
		between[of=H and I, soft clip={domain=0:0.9}];

	\end{axis}
	\end{tikzpicture}
    \caption{\Cref{thm:JMS} defines pairs of lower and upper bounds ($\delta$
        and $\Delta$), such that any degree sequence which obeys these bounds is
        $P$-stable; the area between these functions is filled with vertical
        lines. The pairs $(\delta,\Delta)$ of most distant bounds allowed by
        \cref{eq:JMS} are given by intersections with vertical lines. For
        example, any degree sequence which is (element-wise) between
        $\delta=\frac14 n$ and $\Delta=\frac34 n$ is $P$-stable. In comparison,
        the solid gray region represents a $\sqrt{r}$-wide region around the regular
        degree sequences, which corresponds to the domain of
        \Cref{thm:ak21}.}\label{fig:JMS}
\end{figure}
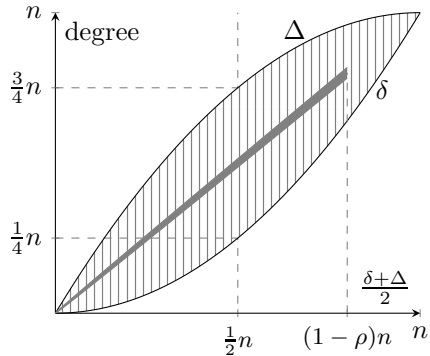

\section{Constructing and bounding the multicommodity-flow}

We will define a number of auxiliary structures. Via these structures, we will
define a multicommodity-flow on the degree-interval Markov chain
$\mathbb{G}(\ell,u)$ and measure its load.

\subsection{Constructing and counting the auxiliary matrices}
\citet{kannan_simple_1999} already introduced an auxiliary matrix to examine the
load of a multicommodity-flow. Our auxiliary matrices will be a little
different. We start with some definitions, then prove two easy statements.
\begin{definition}\label{def:widehatM}
    Let the adjacency matrix of a graph $X$ on vertex set $[n]$ be
    $A_X\in{\{0,1\}}^{n\times n}$. Let $A_{(vw)}$ be the adjacency matrix of the graph $([n],\{vw\})$ with
    exactly one edge. Let us define
    \begin{equation*}
    	\widehat M(X,Y,Z)=A_X+A_Y-A_Z.
    \end{equation*}
\end{definition}
\begin{remark}
If $X,Y,Z$ are graphs on $[n]$, then $\widehat
M(X,Y,Z)\in{\{-1,0,1,2\}}^{n\times n}$.
\end{remark}
Let us define the \emph{matrix switch} operation. (In a previous
paper~\cite{erdos_mixing_2022}, this operation was called a generalized switch.)
\begin{definition}[Switch on a matrix]
    The switch operation on a matrix $M$ on vertices $(a,b,c,d)$ produces the
    matrix
    \begin{equation*}
        M-A_{(ab)}-A_{(cd)}+A_{(ac)}+A_{(bd)}.
    \end{equation*}
\end{definition}
\begin{remark}
    A switch on a graph $X$ corresponds to a switch on its adjacency matrix
    $A_X$.
\end{remark}

\begin{definition}
	Let $M\in {\{-1,0,1,2\}}^{n\times n}$  and let $\deg_M\in \mathbb{Z}^n$ with
    $(\deg_M)_i=\sum_{j=1}^n M_{ij}$ be the sequence of its row sums. We say that $M$ is
	\emph{$c$-tight} (for some $c\in \mathbb{N}$) if $M$ is a symmetric matrix with zero
    diagonal and there exists a graph $W\in \mathcal{G}(\deg_M,\deg_M+\mathds{1}_{\{i,j\}})$ for some
	$\{i,j\}\in\binom{[n]}{2}$ such that $\|M-A_W\|_1\le 2c$.
\end{definition}
Recall the
definition of weak $p$-stability and \cref{eq:wpstable}. We will use the number
of $c$-tight matrices to bound the number of auxiliary matrices.
\begin{lemma}\label{lemma:ctight}
    The number of matrices $M\in {\{-1,0,1,2\}}^{n\times n}$ that are $c$-tight
    and $\deg_M\in [\ell,u]$ for a weakly $p$-stable $[\ell,u]$ is at most
	\begin{equation*}
	\left|\left\{ M\in {\{-1,0,1,2\}}^{n\times n}\ |\ \deg_M\in [\ell,u]\text{\ and $M$ is $c$-tight}\right\}\right|\le n^{2c}\cdot p(n)\cdot
		|\mathcal{G}(\ell,u)|.
	\end{equation*}
\end{lemma}
\begin{proof}
    We can obtain any $c$-tight $M$ we want to enumerate as follows. First select, an appropriate
    $\{i,j\}\in\binom{n}{2}$ and a realization $W\in
    \mathcal{G}(\ell,u+\mathds{1}_{\{i,j\}})$: by weak $p$-stability, there are
    at most $p(n)\cdot|\mathcal{G}(\ell,u)|$ such choices. Then select $c$
    symmetric pairs of positions where the adjacency matrix $A_W$ is changed to
    $-1$ or $+2$ (while preserving symmetry). The latter selection can be made
    in at most $\binom{n}{2}^c\cdot 2^c$ different ways.
\end{proof}

The following lemma is crucial for proving the tightness of the auxiliary matrices
arising in the multicommodity-flow. A switch on a matrix adds $+1$ and
$-1$ to the two-two diagonally opposed entries located in a $2\times 2$ submatrix so that
the row- and column-sums are preserved.
\begin{lemma}[based on Lemma~7.2 of~\cite{erdos_mixing_2022}]\label{lemma:removeplus2s}
	Suppose $M\in {\{-1,0,1,2\}}^{n\times n}$ is a symmetric matrix whose
	diagonal is zero. Suppose further, that
	\begin{enumerate}[label=(\roman*)]
		\item the number of $+2$ entries of $M$ is at most $4$,
		\item the number of $-1$ entries of $M$ is at most $2$,
		\item there exists $V\subseteq [n]$ where $M|_{V\times V}$
			contains every $+2$ and $-1$ entries of $M$,
		\item there exists $v\in V$ such that the $+2$ and $-1$	entries of $M$ are all located in the row and
			column corresponding to $v$,
        \item the row-sum of $v$ in $M|_{V\times V}$ is minimal, and finally,
		\item every row- and column sum in $M|_{V\times V}$ is at least $1$ and
		at most $|V|-2$.
	\end{enumerate}
	Then $M$ is $5$-tight.
\end{lemma}
\begin{proof}
    By Lemma~7.2 of~\cite{erdos_mixing_2022}, there exist
	at most two matrix switches that turn $M$ into a $\{0,1\}$ matrix with the
    possible exception of a symmetric pair of $-1$ entries.
    The $-1$'s remaining after the two matrix switches can be removed by
    adding $+1$ to the pairs of negative entries.
\end{proof}

\subsection{The alternating-trail decomposition}
We consider the set $\binom{[n]}{2}$ in lexicographic order, which
induces an order on the set of edges of any graph defined on $[n]$.
\begin{definition}
    Given a set of edges $\nabla\subseteq \binom{[n]}{2}$ on $[n]$ as vertices,
    let $\nabla_v=\{e\ |\ v\in e\in \nabla\}$. We call $s:\{(v,e)\ |\ v\in e\in
    \nabla\}\to\nabla$ a \emph{pairing function} on $\nabla$ if
    $s(v,\bullet):\nabla_v\to\nabla_v$ defined as $s(v,\bullet):e\mapsto s(v,e)$
    is an involution, i.e., $s(v,\bullet)$ is it own inverse for any $v\in [n]$.
    (The bullet $\bullet$ is the placeholder for the variable $e$ which is the
    second argument of $s$.)
    The set of all pairing functions on $\nabla$ is denoted by $\Pi(\nabla)$.
\end{definition}
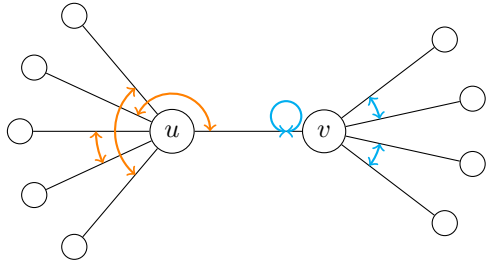
\begin{figure}[H]
	\centering
	\begin{tikzpicture}

	\node[draw,circle] (u) at (2,2) {$u$};
	\node[draw,circle] (v) at (4,2) {$v$};
	\draw (u)--(v);
	\foreach \y in {0,...,4}
	{
		\node[draw,circle] (u\y) at ($ (u)+(130+\y*25:2) $) {};
		\draw (u\y)--(u);
	}
	\foreach \y in {0,...,3}
	{
		\node[draw,circle] (v\y) at ($ (v)+(-37.5+\y*25:2) $) {};
		\draw (v\y)--(v);
	}

	\draw[thick,orange,<->] ($ (u)+(180:1) $) arc (180:205:1);
	\draw[thick,orange,<->] ($ (u)+(130:0.75) $) arc (130:230:0.75);
	\draw[thick,orange,<->] ($ (u)+(155:0.5) $) arc (155:0:0.5);
	
	\draw[thick,cyan,<->] ($ (v)+(180:0.5) $) arc (-90:270:0.2);
	\draw[thick,cyan,<->] ($ (v)+(-37.5:0.75) $) arc (-37.5:-12.5:0.75);
	\draw[thick,cyan,<->] ($ (v)+(12.5:0.75) $) arc (12.5:37.5:0.75);

	\end{tikzpicture}
	\caption{The functions $s(u,\bullet)$ and $s(v,\bullet)$ pair the
	edges incident on $u$ and $v$, respectively.
	The orange arcs (\raisebox{0.1em}{\protect\tikz{\protect\draw[thick,orange](0,0)
	edge[<->,bend left=45] (0.5,0);}}) join edges that are pairs in $s(u,\bullet)$. The
	cyan arcs (\raisebox{0.1em}{\protect\tikz{\protect\draw[thick,cyan](0,0)
	edge[<->,bend left=45] (0.5,0);}}) join edges that are pairs in $s(v,\bullet)$. The
	cyan loop (\raisebox{-0.25em}{\protect\tikz{\protect\draw[thick,cyan,<->](0,0)
	arc (-90:270:0.15);}}) corresponds to the relation $s(v,uv)=uv$.}\label{fig:involution}
\end{figure}

\begin{definition}\label{def:Lns}
    Let $L(\nabla,s)$ be the following subgraph of the line graph  of
    $([n],\nabla)$: join $e,f\in \nabla$ if and only if $e\neq f$ and there
    exists a vertex $v\in e\cap f$ such that $s(v,e)=f$ (or equivalently,
    $s(v,f)=e$).
\end{definition}

\begin{lemma}
    Each connected component of $L(\nabla,s)$ is a path or a cycle.
\end{lemma}
\begin{proof}
    Every edge $e=ij\in \nabla$ has at most two neighbors in $L(\nabla,s)$, the
    edges $s(i,e)$ and $s(j,e)$, thus the maximum degree in $L(\nabla,s)$ is 2.
\end{proof}
\begin{remark}
A cycle in a line graph corresponds to a closed trail in the
original graph. A path in a line graph corresponds to an open trail in the original
graph (which may in theory start and end at the same vertex, but this will never be the
case in our applications, see \Cref{lemma:decomp}). \Cref{def:Lns} generalizes
a concept of~\citet{kannan_simple_1999}, where all of the components are cycles.
\end{remark}
\begin{definition}\label{def:nablapartition}
    Suppose $\nabla\subseteq \binom{[n]}{2}$ and $s$ is a pairing function on
    $\nabla$. Denote by $p_s$ the number of connected components of $L(\nabla,s)$,
    and let us define the unique partition
    \begin{equation}\label{eq:nabla}
        \nabla=W_1^s\uplus W_2^s\uplus \cdots\uplus
        W_{p_s}^s,
    \end{equation}
    where each $W^s_k$ is the vertex set of a component of $L(\nabla,s)$, and
    the sets ${(W^s_k)}_{k=1}^{p_s}$ are listed in the order induced by their
    lexicographically first edges.
\end{definition}

\begin{definition}\label{def:sW}
    For any set of edges $W$ and $s\in \Pi(\nabla)$, let
    \begin{equation*}
        s|_{W}=\left\{ (v,e)\mapsto s(v,e)\ |\ v\in e\in W\text{\ and }s(v,e)\in W\right\}.
    \end{equation*}
    Subsequently, we also define
    \begin{equation*}
        s-e=s|_{\nabla\setminus\{ e\}}.
    \end{equation*}
\end{definition}
\begin{remark}
	If $W$ is the vertex set of a component of $L(\nabla,s)$, then $s|_{W}\in\Pi(W)$.
\end{remark}
\begin{definition}
    If $\nabla=\{u_{i-1}u_{i}\ |\ i=1,\ldots,r\}$ is a set of $r$ distinct edges,
    let $s=u_{0}u_{1}\ldots u_{r-1} u_r\in \Pi(\nabla)$ denote
    \begin{align*}
        s=\bigcup_{1\le i\le r-1}&\Big\{ (u_{i},u_{i-1}u_{i})\mapsto
            u_{i}u_{i+1},\ (u_{i},u_{i}u_{i+1})\mapsto u_{i-1}u_{i}\
        \Big\}\cup \\
         \cup
            &\left\{\begin{array}{ll}
                \left\{(u_0,u_0u_1)\mapsto u_0u_1,(u_r,u_{r-1}u_r)\mapsto
                    u_{r-1}u_r\right\} &\text{if }u_0\neq u_r,\\
                \emptyset &\text{if }u_0=u_r.
            \end{array}\right.
    \end{align*}
\end{definition}
\begin{lemma}
    Let $\nabla=\{u_{i-1}u_{i}\ |\ i=1,\ldots,r\}$ and $s=u_0u_1\ldots u_r$.
    Then
    \begin{itemize}
        \item the walk $u_0u_1\ldots u_r$ is a closed trail if and only if $L(\nabla,s)$ is a cycle, and
        \item the walk $u_0u_1\ldots u_r$ is an open trail if and only if $L(\nabla,s)$ is a path.
    \end{itemize}
    In other words, the Eulerian trails on $\nabla$ can
    be naturally identified with those pairing functions $s\in\Pi(\nabla)$ for
    which $L(\nabla,s)$ is connected.
\end{lemma}
\begin{proof}
    Trivial.
\end{proof}
\Cref{fig:nabla} shows a closed trail defined by a pairing function.
\begin{figure}[ht]
	\centering
    \begin{tikzpicture}[scale=0.7]
		\foreach \d in {0,...,5}
		{
			\node[draw,circle] (k1\d) at (30+60*\d:3) {};
			\node[draw,circle] (k2\d) at ($ (9,0)+(90+60*\d:3) $) {};
		}

		\foreach \d in {0,...,3}
		{
			\pgfmathtruncatemacro{\e}{Mod(2*\d,6)}
			\pgfmathtruncatemacro{\f}{Mod(2*\d+1,6)}
			\pgfmathtruncatemacro{\g}{Mod(2*\d+2,6)}
			\draw[thick,red] (k1\e)--(k1\f) (k2\e)--(k2\f);
			\draw[thick,blue] (k1\f)--(k1\g) (k2\f)--(k2\g);
		}

		\draw[thick,red] (k10)--(k22) (k15)--(k21);
		\draw[thick,blue] (k15)--(k22) (k10)--(k21);

		\draw[<->] ($ (k10)+(150:0.75) $) arc (150:270:0.75);
		\draw[<->] ($ (k10)+(0:0.75) $) arc (0:-39:0.75);

		\draw[<->] ($ (k15)+(90:0.75) $) arc (90:39:0.75);
		\draw[<->] ($ (k15)+(0:0.75) $) arc (0:-150:0.75);

		\draw[<->] ($ (k21)+(-90:0.75) $) arc (-90:-141:0.75);
		\draw[<->] ($ (k21)+(180:0.75) $) arc (180:30:0.75);

		\draw[<->] ($ (k22)+(90:0.75) $) arc (90:-30:0.75);
		\draw[<->] ($ (k22)+(180:0.75) $) arc (180:141:0.75);

		\draw[<->] ($ (k11)+(-30:0.75) $) arc (-30:-150:0.75);
		\draw[<->] ($ (k20)+(-30:0.75) $) arc (-30:-150:0.75);

		\draw[<->] ($ (k12)+(30:0.75) $) arc (30:-90:0.75);
		\draw[<->] ($ (k25)+(-90:0.75) $) arc (-90:-210:0.75);

		\draw[<->] ($ (k13)+(90:0.75) $) arc (90:-30:0.75);
		\draw[<->] ($ (k24)+(-150:0.75) $) arc (-150:-270:0.75);

		\draw[<->] ($ (k14)+(150:0.75) $) arc (150:30:0.75);
		\draw[<->] ($ (k23)+(-210:0.75) $) arc (-210:-330:0.75);
	\end{tikzpicture}
	\caption{An example for $\nabla=\nabla_{X,Y}$ and an
	$({X,Y})$-alternating $s\in \Pi(\nabla)$ (see~\Cref{def:XYalternating}). Red edges belong to
	$X$ and blue edges belong to $Y$. There are $2^4$ different
	$\pi\in\Pi(\nabla)$ that are $({X,Y})$-alternating. There is one
	such $\pi$ where $L(\nabla,\pi)$ has 3 components (the two
	$C_6$'s and a $C_4$ in the middle), and there are 6 cases where
	$L(\nabla,\pi)$ has 2 components. The black arcs represent an
	$s$ such that $L(\nabla,s)$ has exactly one component, or, in
	other words, $s$ defines a closed Eulerian trail on
	$\nabla$.}\label{fig:nabla}
\end{figure}
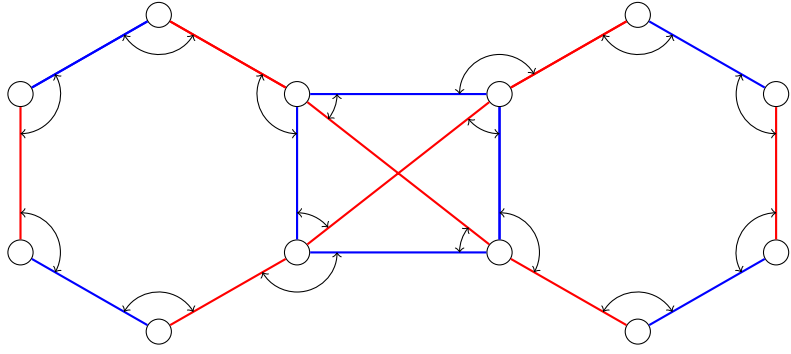

From now on, by slight abuse of notation, we will not distinguish between
$s=u_0u_1\ldots u_r$ as a pairing function and the trail it describes.

\begin{definition}\label{def:alternating}
    Let $Z$ be an arbitrary graph on $n$-vertices and
    let $\nabla\subseteq\binom{[n]}{2}$ be an arbitrary subset of pairs of vertices. A
    pairing-function $s\in \Pi(\nabla)$ is said to be \emph{$Z$-alternating} or
    \emph{alternating in $Z$} if for every $v\in e\in \nabla$ either
    \begin{itemize}
        \item $e$ is a unique solution to $s(v,e)=e$ (the function
            $s(v,\bullet)$ has at most one fixpoint), or
        \item $e\in \nabla\cap E(Z)$ and $s(v,e)\in\nabla\setminus E(Z)$, or
        \item $e\in \nabla\setminus E(Z)$ and $s(v,e)\in\nabla\cap E(Z)$.
    \end{itemize}
    In
    other words, the trail $s|_{W^s_k}$ traverses edges in $Z$ and
    $\overline{Z}$ in turn for any
    $k=1,\ldots,p_s$; furthermore, at any vertex $v\in [n]$, there is at most
    one trail $s|_{W^s_k}$ which starts or ends at $v$. For example, if
    $ij\notin E(Z)$ and $s(i,ij)=ij=s(j,ij)$, then the trail $s|_{\{ij\}}$
    consists of one non-edge of $Z$.

    \medskip

    Furthermore, we say that $s\in \Pi(\nabla)$ is \emph{$Z$-alternating with at most
    $c$ exceptions} if
    \begin{equation}\label{eq:defcexceptions}
        \left|\left\{\{e,s(v,e)\}\ :\ v\in e\in \nabla,\ s(v,e)\neq e\text{ and
            }\Big(\{e,s(v,e)\}\subseteq
    \nabla\cap E(Z)\text{ or }\{e,s(v,e)\}\subseteq\nabla\setminus E(Z)\Big)\right\}\right|\le c
    \end{equation}
    and $s(v,\bullet)$ has at most one fixpoint for every $v\in [n]$. We say
    that $v$ is a site of non-alternation of $s$ in $Z$ if $\{e,s(v,e)\}$ is set
    of size 2 which is a subset of either $\nabla\cap E(Z)$ or $\nabla\setminus
    E(Z)$.
\end{definition}

\begin{definition}\label{def:XYalternating}
    Let $X,Y\in \mathcal{G}(\ell,u)$, where $[\ell,u]$ is a thin degree sequence
    interval. Denote $\nabla_{X,Y}=E(X)\triangle E(Y)$. An $s\in
    \Pi(\nabla_{X,Y})$ which is both {$X$-alternating} and {$Y$-alternating} is
    called \emph{$(X,Y)$-alternating}.
\end{definition}
\begin{lemma}
    Any pairing function $s\in \Pi(\nabla_{X,Y})$ is $X$-alternating if and only
    if $s$ is $Y$-alternating.
\end{lemma}
\begin{proof}
    Trivial, since $\nabla_{X,Y}\setminus E(X)=E(Y)\setminus E(X)$ and
    $\nabla_{X,Y}\cap E(X)=E(X)\setminus E(Y)$.
\end{proof}

\begin{definition}
    Given a degree sequence interval $[\ell,u]$, for any $X,Y\in \mathcal{G}(\ell,u)$, define
    \begin{align*}
        S_{X,Y}=\Big\{ s\in \Pi(\nabla_{X,Y})\ \big|\ s\text{\ is
        $(X,Y)$-alternating}\Big\}.
    \end{align*}
\end{definition}

Recall \Cref{def:sW}. The following \textbf{key decomposition lemma} (KD-lemma) will be referred to repeatedly in this paper.
\begin{lemma}[Key decomposition lemma]\label{lemma:decomp}
    Let $[\ell,u]$ be a thin degree sequence interval, and let $X,Y\in
    \mathcal{G}(\ell,u)$, $s\in S_{X,Y}$. Then $s|_{W^s_k}$ is
    {$(X,Y)$-alternating}, and $s|_{W^s_k}$ describes an Eulerian trail on
    $W^s_k$ for any $1\le k\le p_s$. If $s|_{W^s_k}$ describes an open trail,
    then its end-vertices are (by definition) distinct, and the end-vertices of
    the trail $s|_{W^s_k}$ are disjoint from the end-vertices of any other open trail
    $s|_{W^s_j}$ ($j\neq k$).
\end{lemma}
\begin{proof}
    We have $|\deg_X(v)-\deg_Y(v)|\le 1$ for any $v\in V$. Thus the involution
    $s(v,\bullet)$ pairs the $X$-edges of $\nabla_{X,Y}$ incident to $v$ to the $Y$-edges of
    $\nabla_{X,Y}$ incident to $v$, with the exception of the at most one
    fixpoint of $s(v,\bullet)$. The closed trails must have even length, because
    $s(v,\bullet)$ pairs $X$-edges to $Y$-edges at any $v$.

    Clearly, if an open trail $s|_{W^s_i}$ both starts and ends at $v$, then $s(v,\bullet)$ has
    at least two fixpoints, which is a contradiction. Similarly, we have a
    contradiction if more than one trail terminates at some vertex $v$. Lastly,
    if $s|_{W^s_k}$ is an open trail, then the degree $\deg_{W^s_k}(v)$ is even, except
    if $v$ is one of the two end-vertices of $s|_{W^s_k}$, in which case
    $\deg_{W^s_k}(v)$ is odd.
\end{proof}

\begin{lemma}\label{lemma:flowamount}
    For any thin degree sequence interval $[\ell,u]$ on $n$ vertices and any two graphs
	$X,Y\in\mathcal{G}(\ell,u)$
    \begin{equation}\label{eq:flowamount}
		|S_{X,Y}|=\prod_{v\in
        [n]}\left\lceil\frac{\deg_{\nabla_{X,Y}}(v)}{2}\right\rceil\text{\Large\rm
    !},
	\end{equation}
    where the right hand side is the product of factorials.
\end{lemma}
\begin{proof}
We have
\begin{equation*}
|\deg_{E(X)\setminus E(Y)}(v)-\deg_{E(Y)\setminus
E(X)}(v)|=|\deg_X(v)-\deg_Y(v)|\le 1.
\end{equation*}
If $\deg_X(v)=\deg_Y(v)$, then we have
$(\deg_X(v)-\deg_{X\cap Y}(v))!=(\frac12\deg_{\nabla_{X,Y}}(v))!$ ways to choose
$s(v,\bullet)$ such that it is an involution which maps edges of $X$ to
edges of $Y$: if $s(v,\bullet)$ had a fixpoint, then by parity it must have had another,
too, which contradicts \Cref{def:alternating}.

If $\deg_X(v)=\deg_Y(v)+1$ and $s(v,e)=e$, then $e\in E(X)\setminus E(Y)$ and
$e$ is the only fixpoint of $s(v,\bullet)$. Therefore there are
$\deg_X(v)-\deg_{X\cap Y}(v)=\frac12(\deg_{\nabla_{X,Y}}(v)+1)$ ways to choose the fixpoint, and
$(\deg_X(v)-\deg_{X\cap Y}(v)-1)!$ ways to choose the rest of the map $s(v,\bullet)$.
\end{proof}

\begin{lemma}\label{lemma:flowcexceptions}
    For any graph $Z\in\mathcal{G}(\ell,u)$ for a thin degree sequence interval
    $[\ell,u]$ and any $\nabla\subseteq\binom{[n]}{2}$, we have
    \begin{equation}\label{eq:cexceptions}
        \left|\Big\{ s\in \Pi(\nabla)\ \big|\ s\text{\ is $Z$-alternating with at most
        $c$ exceptions}\Big\}\right|\le n^{3c}\cdot\prod_{v\in
    [n]}\left\lceil\frac{\deg_{\nabla}(v)}{2}\right\rceil\text{\Large\rm !}
	\end{equation}
\end{lemma}
\begin{proof}
    There are at most $n^{3c}$ different choices for the set on the left hand
    side of \cref{eq:defcexceptions}. If we fix the non-alternating pairs, then
    the number of remaining choices at $s(v,\bullet)$ are still upper bounded by
    $\lceil\frac12\deg_{\nabla}(v)\rceil{!}$, thus~\cref{eq:cexceptions} holds.
\end{proof}

\subsection{The precursor}\label{sec:precur}

So far, every proof of rapid mixing for the switch Markov chain which is based
on Sinclair's method contains at its core a counting
lemma~(\citet{greenhill_switch_2014}). The purpose of the counting lemma is to
enumerate the possible auxiliary structures and parameter sets from which the
source and sink of any commodity passing through a realization $Z$ can be
recovered from. The difficult technical parts of the proofs are concerned with
the maintenance and upkeep associated to these structures. To our surprise, for
thin degree sequence intervals, by slightly tweaking these structures, the
arising technicalities can almost entirely be reduced
to~\cite{erdos_mixing_2022}, and a major shortcut is taken by this paper by
reusing these parts. A relatively long, but mostly elementary
\Cref{def:precursor} will specify the properties that we expect from the
auxiliary structures and parameter sets borrowed from~\cite{erdos_mixing_2022}.
In \Cref{sec:construction}, we will use this framework to recombine the borrowed
parts into a proof for thin degree sequence intervals.

\medskip

The decomposition in \Cref{def:nablapartition} is formally very similar to the
decomposition in~\cite[Section~4.1]{erdos_mixing_2022}. Whenever the degree
sequences of $X$ and $Y$ are identical ($\nabla=\nabla_{X,Y}$ and $s\in
S_{X,Y}$), the two decompositions are actually identical. In any other case, for
every two unit differences between the degree sequences of $X$ and $Y$ we will
utilize a hinge-flip or an edge-toggle in the multicommodity-flow between $X$
and $Y$.

\medskip

Let us now turn to defining the framework for the reduction
to~\cite{erdos_mixing_2022}. We need the following structure and in particular
the matrix $M$ to be able to find an appropriate reduction which is compatible
with the processes of~\cite{erdos_mixing_2022}.
\begin{definition}\label{def:DM}
    Let $M\in {\{0,1,2\}}^{n\times n}$ be a symmetric matrix with zero diagonal.
    For technical purposes, let us define the following set of triples:
    \begin{equation}\label{eq:DM}
        \mathfrak{D}_M =\left\{ (X,Y,s)\ \Bigg|\
            \text{\ where\ }
            \begin{cases}
                X,Y\text{ graphs on }[n],\ s\in S_{X,Y},\\
                \left\{vw\ |\ v\neq w,\ M_{vw}=2\right\}\subseteq E(X)\cap E(Y),\\
                \left\{vw\ |\ v\neq w,\ M_{vw}=0\right\}\subseteq E(\overline
                X)\cap E(\overline Y).
           \end{cases} \right\}.
    \end{equation}
\end{definition}
The next definition collects a number of properties (of the multicommodity-flow
and the auxiliary structures designed for the switch Markov chain) that we want
to preserve from~\cite{erdos_mixing_2022}.
\begin{definition}\label{def:precursor}
    We call the ordered triple $(\varUpsilon,B,\pi)$ a \emph{precursor with parameter
    $c\in \mathbb{N}$}, if the following properties hold. The objects $\varUpsilon_M$, $B_M$, and
    $\pi_M$ are functions for any symmetric matrix $M\in
    {\{0,1,2\}}^{n\times n}$ with zero diagonal, where $n\in \mathbb{N}$. We require that the domain of $\varUpsilon_M$
    satisfies
    \begin{equation}
        \dom(\varUpsilon_{M})\subseteq \mathfrak{D}_M,\label{eq:domUpsilon}
    \end{equation}
    Furthermore, for any $(X,Y,s)\in \dom(\varUpsilon_M)$, let us define two
    degree sequences:
    \begin{align*}
        \ell_{X,Y}&={\Big(\min\{\deg_X(i),\deg_Y(i)\}\Big)}_{i=1}^n,\\
        u_{X,Y}&={\Big(\max\{\deg_X(i),\deg_Y(i)\}\Big)}_{i=1}^n.
    \end{align*}
    We require that $\varUpsilon_{M}(X,Y,s)$ is a sequence of
	graphs that forms a path connecting $X$ and $Y$ in the Markov graph
	$\mathbb{G}(\ell_{X,Y},u_{X,Y})$. We require that $\pi_M$ and $B_M$ is
    defined on
    \begin{equation*}
            \dom(\pi_M)=\dom(B_M)=\{ (X,Y,s,Z)\ |\
			Z\in\varUpsilon_M(X,Y,s),\ (X,Y,s)\in\dom(\varUpsilon_M)\}.
    \end{equation*}
    Moreover:
    \begin{enumerate}[label={(\alph*)}]
		\item The length of $\varUpsilon_M(X,Y,s)$ is at most
			$c\cdot|\nabla_{X,Y}|$.\label{prop:length}
		\item The size $|(E(X)\triangle E(Z))\setminus \nabla_{X,Y}|\le c$ for any
            $Z\in\varUpsilon_M(X,Y,s)$.\label{prop:cexceptions}
		\item The matrix $M-A_Z$ is $c$-tight for any $Z\in
			\varUpsilon_{M}(X,Y,s)$.\label{prop:ctight}
		\item The pairing function $\pi_M(X,Y,s,Z)$ is a member of
            $\Pi(\nabla_{X,Y})$ and it is alternating in $Z$
			with at most $c$ exceptions.\label{prop:picexceptions}
        \item $\pi_M(X,Y,s,X)=\pi_M(X,Y,s,Y)=s$.\label{prop:pistartend}
        \item If $L(\nabla_{X,Y},s)$ is connected then
            $L(\nabla_{X,Y},\pi_M(X,Y,s,Z))$ is also
            connected.\label{prop:connected}
        \item The cardinality of
            \begin{equation*}
                \mathfrak{B}_n=\Big\{B_M(X,Y,s,Z)\ \mid\ Z\in\varUpsilon_M(X,Y,s),\
	        	\text{$M$ arbitrary},\ (X,Y,s)\in\dom(\varUpsilon_M),\
            |V(X)|=n\Big\}
            \end{equation*}
        is at most a constant times $n^c$, i.e., $|\mathfrak{B}_n|=\mathcal{O}(n^c)$.\label{prop:Bpoly}
		\item The function\label{prop:Psi}
			\begin{align*}
                \varPsi=\left\{(Z,\nabla_{X,Y},\pi_M(X,Y,s,Z),B_M(X,Y,s,Z))\mapsto
                    (X,Y,s)\ \big|\ \text{$M$ is arbitrary},\ Z\in\varUpsilon_M(X,Y,s) \right\}
            \end{align*}
            is well-defined, i.e., two different images in the co-domain are not
            assigned to the same element from the domain of
            $\varPsi$.\hfill\qedsymbol
    \end{enumerate}
\end{definition}
Typically, the value of $B_M(X,Y,s,Z)$ will be a long tuple (an ordered set of
parameters).
The exact value of $c$ is not important here, the requirements only impose a
lower bound on its value. However, it is important to note that $c$ is a constant,
independent even from the number of vertices $n$. Note also that in applications of
\Cref{def:precursor}, the matrix $M$ will not be completely arbitrary.

\begin{definition}
    A subset $\mathfrak{P}$ is a \emph{precursor domain} if it is a set of
    triples $(X,Y,s)$ such that $X$ and $Y$ have the same vertex set $[n]$ for
    some $n\in\mathbb{N}$ (where $n$ may vary) and $s\in S_{X,Y}$. We say that a
    precursor $(\varUpsilon,B,\pi)$ is defined on a \emph{precursor domain}
    $\mathfrak{P}$ if and only if for any $n\in \mathbb{N}$ and symmetric matrix
    $M\in {\{0,1,2\}}^{n\times n}$ with zero diagonal we have
    \begin{equation*}
        \dom(\varUpsilon_{M})\supseteq \mathfrak{P}\cap \mathfrak{D}_M.
    \end{equation*}
\end{definition}
Let us define two precursor domains:
\begin{align}
    \mathfrak{C}_\mathrm{thin}&=\Big\{ (X,Y,s)\ \Big|\ s\in S_{X,Y},\
L(\nabla_{X,Y},s)\text{\ is connected, and }\|\deg_X-\deg_Y\|_\infty\le
1\Big\},\label{eq:Cthin}\\
    \mathfrak{R}_\mathrm{thin}&=\Big\{ (X,Y,s)\ \Big|\ s\in S_{X,Y}\text{\ and
    }\|\deg_X-\deg_Y\|_\infty\le 1\Big\}.\label{eq:Rthin}%
\end{align}

The set $\mathfrak{C}_\mathrm{thin}$ describes the identifiers of the small parts
from which the whole multicommodity-flow will be built from. In contrast,
the multicommodity-flow was built in~\cite{erdos_mixing_2022} for each
triple in $\mathfrak{R}_\mathrm{thin}$ directly.

\begin{lemma}\label{lemma:extension}
    If there exists a precursor with parameter $c$ which is defined on
    $\mathfrak{C}_\mathrm{thin}$, then there exists a precursor on
    $\mathfrak{R}_\mathrm{thin}$ with parameter $3c$.
\end{lemma}
\begin{proof}
	We will show that the precursor can be extended so that it is also defined
    on $\mathfrak{R}_\mathrm{thin}$ without violating \Cref{def:precursor}.
    For any $(X,Y,s)\in \mathfrak{R}_\mathrm{thin}\cap \mathfrak{D}_M$, we construct a path in the
    Markov graph of $\mathbb{G}(\ell_{X,Y},u_{X,Y})$, where $[\ell_{X,Y},u_{X,Y}]$ is the smallest
    degree sequence interval that contains both $\deg_X$ and $\deg_Y$.
    By the thinness of $\mathcal{I}$, we have $|\ell(i)-u(i)|\le 1$ for every
    $i\in [n]$.
    According to \Cref{def:nablapartition} and the KD-lemma
    (\Cref{lemma:decomp}), any $s\in S_{X,Y}$ partitions
    $\nabla_{X,Y}=E(X)\triangle E(Y)$ into edge sets of {$(X,Y)$-alternating}
    trails, let that decomposition be
    \begin{equation*}
        \nabla_{X,Y}=W_1^s\uplus W_2^s\uplus\cdots\uplus W_{p_s}^s.
    \end{equation*}
    Let
    \begin{equation*}
        G^{X,Y}_k=X\triangle \bigcup_{i=1}^{k}W_i^s,
    \end{equation*}
    so that $G^{X,Y}_0=X$ and $G^{X,Y}_{p_s}=Y$.
    By definition, $s|_{W^s_k}$ is connected, so $(G^{X,Y}_{k-1},G^{X,Y}_{k},s|_{W^s_k})\in
    \mathfrak{C}_\mathrm{thin}$ for $k=1,\ldots,p_s$. Let us confirm that
    $G^{X,Y}_k\in \mathcal{G}(\ell,u)$. If $s_k$ is a closed trail, then the
    degree sequences of $G^{X,Y}_k$ and $G^{X,Y}_{k-1}$ are identical. If $s_k$
    is an open trail whose end-vertices are $v$ and $w$, then the degree
    sequences of $G^{X,Y}_{k}$ and
    $G^{X,Y}_{k-1}$ differ by 1 precisely on $v$ and $w$; since
    these end-vertices are distinct from any other end-vertices of another open
    trail $s_j$, such a change of the degree of $v$ and $w$ not occur for any
    other $k$. Thus the degree $v$ satisfies:
    \begin{equation}\label{eq:ext_degv}
        \deg_{G^{X,Y}_i}(v)\in\{\deg_X(v),\deg_Y(v)\}\text{ for any
        }i=1,\ldots,p_s\text{\ and any }v\in [n],
    \end{equation}
    and so $\deg_{G^{X,Y}_i}\in [\ell_{X,Y},u_{X,Y}]$.

    \medskip

    It is easy to see that
    $(G^{X,Y}_{k-1},G^{X,Y}_{k},s|_{W^s_k})\in\mathfrak{D}_M$. If $e\in E(X)\cap
    E(Y)$, then $e\notin W^s_i$ for any $i$. Similarly, if $e\in
    E(\overline{X})\cap E(\overline{Y})$, then $e\notin W^s_i$ for any $i$.

    \medskip

    We may now define
    $\varUpsilon_M$ on $\mathfrak{R}_\mathrm{thin}$ recursively: concatenate the
    sequences $\varUpsilon_M(G^{X,Y}_{k-1},G^{X,Y}_k,s|_{W_k^s})$ in increasing
    order of $k$ to obtain
    \begin{equation}
        \varUpsilon_M(X,Y,s)={\left(\varUpsilon_M\left(G^{X,Y}_{k-1},G^{X,Y}_k,s|_{W_k^s}\right)\right)}_{k=1}^{p_s},
    \end{equation}
    where the concatenation keeps only one of the last and first element of consecutive sequences. For $Z\in
    \varUpsilon_M\left(G^{X,Y}_{k-1},G^{X,Y}_k,s|_{W^s_k}\right)$ (take the maximal $k$ such
    that the relation holds) let
    \begin{align}
		\pi_M(X,Y,s,Z)&=\left(\bigcup_{i=1}^{k-1}s|_{W^s_i}\right)\cup
        \pi_M\left(G^{X,Y}_{k-1},G^{X,Y}_k,s|_{W^s_k},Z\right)\cup\left(\bigcup_{i=k+1}^{p_s}
        s|_{W^s_i}\right)\label{eq:piextend}\\
		B_M(X,Y,s,Z)&=\left(k-1,B_M\left(G^{X,Y}_{k-1},G^{X,Y}_k,s|_{W^s_k},Z\right)\right)\label{eq:Bextend}
    \end{align}

    We claim that the extended functions provide a precursor on
    $\mathfrak{R}_\mathrm{thin}$. Let us check the non-trivial properties of
    \Cref{def:precursor}. Suppose that $Z\in
    \varUpsilon_M\left(G^{X,Y}_{k-1},G^{X,Y}_k,s|_{W^s_k}\right)$. Then $\deg_Z\in
    [\ell_{X,Y},u_{X,Y}]$, since $\deg_{G^{X,Y}_{k-1}},\deg_{G^{X,Y}_{k}}\in
    [\ell_{X,Y},u_{X,Y}]$.

    \emph{Checking \Cref{def:precursor}\ref{prop:cexceptions}.}
    By \Cref{def:precursor}\ref{prop:cexceptions},
    $\left|\left(E\left(G^{X,Y}_{k-1}\right)\triangle E(Z)\right)\setminus
    W^s_k\right|\le c$, and
    \begin{equation}\label{eq:nonnablachanged}
        (E(X)\triangle E(Z))\setminus \nabla_{X,Y}=\left(E\left(G^{X,Y}_{k-1}\right)\triangle
        E(Z)\triangle \bigcup_{i=1}^{k-1} W^s_i\right)\setminus
        \nabla_{X,Y}\subseteq \left(E\left(G^{X,Y}_{k-1}\right)\triangle
        E(Z)\right)\setminus
        W^s_k,
    \end{equation}
    therefore the LHS has cardinality at most $c$ as
    well.\hfill$\qedsymbol_{\ref{prop:cexceptions}}$

    \medskip

    \emph{Checking \Cref{def:precursor}\ref{prop:ctight}.} The precursor
    property holds for $\mathfrak{C}_\mathrm{thin}$ and
    $\left(G^{X,Y}_{k-1},G^{X,Y}_k,s_k\right)\in\mathfrak{C}_\mathrm{thin}\cap
    \mathfrak{D}_M$,
    therefore $M-A_Z$ is $c$-tight.

    \medskip

    \emph{Checking \Cref{def:precursor}\ref{prop:picexceptions}.}
    By~\cref{eq:nonnablachanged}, $\pi_M(X,Y,s,Z)$
    alternates in $Z$ with at most $2c$ extra exceptions on top of the $c$
    non-alternations of $\pi_M\left(G^{X,Y}_{k-1},G^{X,Y}_k,s|_{W^s_k},Z\right)$.
    \hfill$\qedsymbol_{\ref{prop:picexceptions}}$

    \medskip

    \emph{Checking \Cref{def:precursor}\ref{prop:Bpoly}.}
    The cardinality of the range of $B_M(X,Y,s,Z)$ grows by a factor of at most
    $n^2$, due to the one extra integer $k-1$.
    \hfill$\qedsymbol_{\ref{prop:Bpoly}}$

    \medskip

    \emph{Checking \Cref{def:precursor}\ref{prop:Psi}.}
    The last missing piece to proving that the extended functions are a
    precursor on $\mathfrak{R}_\mathrm{thin}$ is showing that $\varPsi$ is
    well-defined on the larger domain. By
    \Cref{def:precursor}\ref{prop:connected}, the connected components of
    $L(\nabla_{X,Y},\pi_M(X,Y,s,Z))$ determine $W^s_i$ and $s|_{W^s_i}$ for any
    $i\neq k$, and also $\pi_M(G^{X,Y}_{k-1},G^{X,Y}_k,s|_{W^s_k},Z)$ and
    $W^s_k$.

    \medskip

    The number $k-1$ is recorded in $B_M(X,Y,s,Z)$ (see \cref{eq:Bextend}), which
    is an argument of $\varPsi$. Knowing $k$, we can select
    $\pi_M(G^{X,Y}_{k-1},G^{X,Y}_k,s|_{W^s_k},Z)$ from the components of
    $L(\nabla_{X,Y},\pi_M(X,Y,s,Z))$, see \cref{eq:piextend}. Because the
    original functions provide a precursor on
    $\mathfrak{C}_\mathrm{thin}$, the original $\varPsi$ function is well-defined,
    so the following value
    \begin{equation*}
		\varPsi\Bigg(Z,W^s_k,\pi_M\left(G^{X,Y}_{k-1},G^{X,Y}_k,s|_{W^s_k},Z\right),
		B_M\left(G^{X,Y}_{k-1},G^{X,Y}_k,s|_{W^s_k},Z\right)\Bigg)=
		\left(G^{X,Y}_{k-1},G^{X,Y}_k,s|_{W^s_k}\right).
    \end{equation*}
    is determined.
    Since
    \begin{align*}
        E(X)=E\left(G^{X,Y}_{k-1}\right)\triangle\bigcup_{i=1}^{k-1}W^s_i,\qquad\qquad
        E(Y)=E(X)\triangle \nabla_{X,Y},\qquad\qquad s=\bigcup_{i=1}^{p_s} s|_{W^s_i},
    \end{align*}
    we have shown that $\varPsi$ is well-defined even on the extended
    domain. %
\end{proof}
In the proof of \Cref{lemma:extension}, we extensively used the fact that the degree
sequence intervals in $\mathcal{I}$ are thin.
\begin{theorem}\label{thm:precursor}
    Let $\mathcal{I}$ be a set of weakly $P$-stable degree sequence intervals.
    If there exists a precursor on $\mathfrak{R}_\mathrm{thin}$ with parameter
    $c$ then the degree interval Markov chain $\mathcal{G}(\ell,u)$ is rapidly
    mixing for any $[\ell,u]\in\mathcal{I}$.
\end{theorem}
\begin{proof}
    This proof is not new and fairly straightforward, but it is presented for
    the sake of completeness. The core of this approach had already appeared in
    the paper of~\citet{kannan_simple_1999}. We will practically repeat the
    skeleton of the proof of~\cite{erdos_mixing_2022} using the definitions of
    the precursor, which hides the majority of the technical difficulties. 
    We will take $M=A_X+A_Y$, but we want to be explicit about the dependence on
    $X$ and $Y$ even when $M$ appears as an index, so let $X+Y$ denote the
    matrix $A_X+A_Y$ in this proof.

    \medskip

    Let $[\ell,u]\in\mathcal{I}$, where $\ell$ and $u$ are degree sequences on
    $[n]$. Let us define the multicommodity-flow $f$ on the Markov-graph of
    $\mathbb{G}(\ell,u)$: for every $X,Y\in \mathcal{G}(\ell,u)$ and $s\in
    S_{X,Y}$, send ${\sigma(X)\sigma(Y)}/{|S_{X,Y}|}$ amount of flow on
    $\varUpsilon_{X+Y}(X,Y,s)$. The total flow in $f$ from $X$ to $Y$ sums to
    $\sigma(X)\sigma(Y)$.
	
	\medskip

    Let us recall~\cref{eq:sinclair}:
    \begin{equation}\label{eq:mixingtime1}
        \tau_{\mathbb{G}(\ell,u)}(\varepsilon)\le \rho(f)\,\ell(f)
        \left(\log|\mathbb{G}(\ell,u)|-\log \varepsilon\right)\le \rho(f)\,
        \ell(f) \left(\binom{n}{2}-\log \varepsilon\right).
    \end{equation}
    By \Cref{def:precursor}\ref{prop:length}, $\ell(f)\le c\cdot\binom{n}{2}$.
    It only remains to show that $\rho(f)$ is polynomial in $n$. Continuing
    \cref{eq:rho} with the substitution $\mathbb{G}=\mathbb{G}(\ell,u)$:
    \begin{align}
        \rho(f)=&\max_{ZW\in E(\mathbb{G})}\frac{1}{\sigma(Z)\Pr_\mathbb{G}(Z\to W)}
        \sum_{\substack{ X,Y\in \mathcal{G}(\ell,u),\ s\in S_{X,Y}\\ ZW\in
        E(\varUpsilon_{X+Y}(X,Y,s))}}f(\varUpsilon_{X+Y}(X,Y,s))\nonumber\\
        \rho(f)\le &\max_{ZW\in E(\mathbb{G})}\frac{1}{\sigma(Z)\cdot 6/n^4}
        \sum_{\substack{ X,Y\in \mathcal{G}(\ell,u),\ s\in S_{X,Y}\\ ZW\in
        E(\varUpsilon_{X+Y}(X,Y,s))}}\frac{\sigma(X)\sigma(Y)}{|S_{X,Y}|}\nonumber\\
        \rho(f)\le &\frac{n^4}{6|\mathcal{G}(\ell,u)|}\cdot\max_{ZW\in
        E(\mathbb{G})}
        \sum_{\substack{ X,Y\in \mathcal{G}(\ell,u),\ s\in S_{X,Y}\\ ZW\in
        E(\varUpsilon_{X+Y}(X,Y,s))}}\frac{1}{|S_{X,Y}|}\nonumber\\
        \rho(f)\le &\frac{n^4}{6|\mathcal{G}(\ell,u)|}\cdot\max_{Z\in
        V(\mathbb{G})}
        \sum_{\substack{ X,Y\in \mathcal{G}(\ell,u),\ s\in S_{X,Y}\\ Z\in
        \varUpsilon_{X+Y}(X,Y,s)}}\frac{1}{|S_{X,Y}|}\label{eq:mixingtime2}
    \end{align}
    According to \Cref{def:precursor}\ref{prop:Psi}, given $Z$, $\nabla_{X,Y}$,
    $\pi_{X+Y}(X,Y,s,Z)$, and $B_{X+Y}(X,Y,s,Z)$, the function $\varPsi$
    determines $(X,Y,s)$. Therefore the relation $Z\in \varUpsilon_{X+Y}(X,Y,s)$ is
    equivalent to saying that there exists a triple $(\nabla,\pi,B)$ such that
    $(Z,\nabla,\pi,B)\in \varPsi^{-1}(X,Y,s)$:
    \begin{equation}
        \rho(f)\le \frac{n^4}{6|\mathcal{G}(\ell,u)|}\cdot\max_{Z\in
        V(\mathbb{G})}\sum_{\substack{(Z,\nabla,\pi,B)\in \varPsi^{-1}(X,Y,s)}}\frac{1}{|S_{X,Y}|}\label{eq:mixingtime3}
    \end{equation}
    Next, we use \Cref{lemma:flowamount}, which shows that $|S_{X,Y}|$ is
    determined by $\nabla_{X,Y}$, and its value does not depend directly on $X$
    or $Y$.
    \begin{equation}
        \rho(f)\le \frac{n^4}{6|\mathcal{G}(\ell,u)|}\cdot\max_{Z\in
        V(\mathbb{G})}\sum_{\substack{(Z,\nabla,\pi,B)\in
    \varPsi^{-1}(X,Y,s)}}\quad\prod_{v\in
        [n]}{\left(\left\lceil\frac{\deg_{\nabla}(v)}{2}\right\rceil\text{\Large\rm
    !}\right)}^{-1}\label{eq:mixingtime4}
    \end{equation}
    Given $Z$, the matrix $\widehat M(X,Y,Z)$ (\Cref{def:widehatM}) determines
    $\nabla_{X,Y}=E(X)\triangle E(Y)$: the edges
    that belong to $\nabla_{X,Y}$ are precisely those where the sum of the
    adjacency matrices $A_X+A_Y=\widehat M(X,Y,Z)+A_Z$ takes 1. Furthermore, by
    a property of the precursor, for
    any $Z\in \varUpsilon_{X+Y}(X,Y,s)$, we have $\deg_Z\in
    [\ell_{X,Y},u_{X,Y}]$, therefore
    \begin{equation*}
        \deg_{\widehat M(X,Y,Z)}=\deg_{A_X+A_Y-A_Z}=\deg_X+\deg_Y-\deg_Z\in [\ell_{X,Y},u_{X,Y}]\subseteq
        [\ell,u].
    \end{equation*}
    Now using that $\widehat M(X,Y,Z)$ is $c$-tight, it follows from
    \Cref{lemma:ctight,lemma:flowcexceptions} that
    \begin{equation}
        \rho(f)\le \frac{n^4}{6|\mathcal{G}(\ell,u)|}\cdot\max_{Z\in
        V(\mathbb{G})}\sum_{B\in \mathfrak{B}_n} n^{5c}\cdot p(n)\cdot |\mathcal{G}(\ell,u)|\le
        \frac16n^{5c+4}\cdot p(n)\cdot |\mathfrak{B}_n|,\label{eq:mixingtime5}
    \end{equation}
    where the right hand side is dominated by a polynomial of $n$
    (according to \Cref{def:precursor}\ref{prop:Bpoly}). In conclusion, the mixing time in
    \cref{eq:mixingtime1} is polynomial.
\end{proof}

To prove \Cref{thm:main}, it only remains to construct a precursor on
$\mathfrak{C}_\mathrm{thin}$. The next section proceeds with the construction in
two separate stages.

\section{Constructing the precursor}\label{sec:construction}

We will construct a precursor on $\mathfrak{C}_\mathrm{thin}$ for any weakly
$P$-stable thin set of degree sequence intervals $\mathcal{I}$ in two stages.
In the first stage, we show that there exists a precursor on
$\mathfrak{C}_\mathrm{id}$ (see \Cref{def:Ct}),
and then we will extend this precursor to $\mathfrak{C}_\mathrm{thin}$ in the
second stage. Then we will apply \Cref{lemma:extension} and \Cref{thm:precursor}
to prove \Cref{thm:main}.

\subsection{Stage 1: closed trails}
\begin{definition}\label{def:Ct} Let us define
\begin{align*}
    \mathfrak{C}_\mathrm{id}&=\Big\{ (X,Y,s)\ \Big|\ s\in S_{X,Y},\
L(\nabla_{X,Y},s)\text{\ is connected, and }\deg_X=\deg_Y\Big\},\\
\mathfrak{R}_\mathrm{id}&=\Big\{ (X,Y,s)\ \Big|\ s\in S_{X,Y}\text{\ and }
        \deg_X=\deg_Y\Big\}.
\end{align*}
\end{definition}
The graph $L(\nabla_{X,Y},s)$ is a cycle for any $(X,Y,s)\in
\mathfrak{C}_\mathrm{id}$, because the degree sequences of $X$ and
$Y$ are identical. To handle this case, a large machinery was developed
in~\cite{erdos_mixing_2022}. However, there the range of auxiliary matrices $M$
was much smaller. Because of the larger range of auxiliary matrices in the
current paper, we had to introduce and explicitly define the precursor.
Therefore, we unfortunately need to repeat some parts of the proof
of~\cite{erdos_mixing_2022} to obtain those claims in the desired generality.
The following lemma collects the necessary technical lemmas proved
in~\cite{erdos_mixing_2022}.
\begin{lemma}\label{lemma:Cthin}
    There exists a precursor on $\mathfrak{C}_\mathrm{id}$ with
    parameter $c=12$.
\end{lemma}
\emph{Proof.} Let $(X,Y,s)\in\mathfrak{C}_\mathrm{id}$ be arbitrary
with $X,Y\in\mathcal{G}(d)$. Since $s$ is an ${(X,Y)}$-alternating closed trail,
$|\nabla_{X,Y}|$ is even.
In~\cite{erdos_mixing_2022}, the path $\varUpsilon(X,Y,s)$ in the switch Markov
graph is defined exactly when the degree sequences of $X$ and $Y$ are identical
and $s\in S_{X,Y}$. We use the definition of $\varUpsilon(X,Y,s)$
from~\cite{erdos_mixing_2022} only when $s\in S_{X,Y}$ and $L(\nabla_{X,Y},s)$
is a cycle, so when $p_s=1$.

\medskip

First of all, let us recall that $\varUpsilon(X,Y,s)$
in~\cite{erdos_mixing_2022} describes a sequence of graphs such that each two
consecutive graphs can be obtained from each other by a switch.
In~\cite[Definition~4.2]{erdos_mixing_2022}, for any $s\in S_{X,Y}$, the path
$\varUpsilon(X,Y,s)$ is composed by concatenating a number of \textsc{Sweep}
sequences:
\begin{equation*}
	\varUpsilon(X,Y,s)=
	{\left({\left(\textsc{Sweep}(G^k_r,C^k_r)\right)}_{r=1}^{\mu_{k}+1} \right)}_{k=1}^{p_s},
\end{equation*}
where $C^k_r$ are circuits and $G^k_{r+1}=G^k_r\triangle C^k_r$, where
$G^k_0=X\triangle \uplus_{i=1}^{k-1} W^s_i$ and $G^k_{\mu_k+1}=X\triangle
\uplus_{i=1}^{k} W^s_i$, and $W^s_k=\uplus_{r=1}^{\mu_k+1} E(C^k_r)$ and
$\nabla_{X,Y}=\uplus_{k=1}^{p_s} W^s_k$. It is easy to check
in~\cite[Algorithm~2.1]{erdos_mixing_2022}, that $\textsc{Sweep}(G^k_r,C^k_r)$
is a sequence of switches such that each switch is incident with all four
vertices on $V(C^k_r)$.

\medskip

When $(X,Y,s)\in\mathfrak{C}_\mathrm{id}\cap \mathfrak{D}_M$,
by definition $L(\nabla_{X,Y},s)$ is connected and $p_s=1$, thus we may define
\begin{equation}\label{eq:defUpsilon}
	\varUpsilon_{M}(X,Y,s)={\left(\textsc{Sweep}(X\triangle
    \uplus_{i=1}^{r-1}C_i,C_r)\right)}_{r=1}^{\mu+1},
\end{equation}
where $s$ decomposes $\nabla_{X,Y}$ into (primitive) circuits ${(C_r)}_{r=1}^\mu$ such
that $\nabla=\uplus_{r=1}^\mu E(C_r)$, see~\cite[Lemma~5.13]{erdos_mixing_2022}.
The circuit $C_r$ defines a cyclical order on its vertices, but $\textsc{Sweep}$
takes a linear order, so we still need to select the \textbf{cornerstone}, where
the linear order starts to enumerate the vertices in the given cyclical order.
The choice of the cornerstone (\cite[eq.~(5.11)]{erdos_mixing_2022}) only
plays a role in proving that $\widehat M(X,Y,Z)$ is close to the adjacency
matrix of an appropriate graph in $\ell_1$-norm. From the rest of
\cite{erdos_mixing_2022}'s point of view, the cornerstone is arbitrarily chosen.

\medskip

In this adaptation of the proof in~\cite{erdos_mixing_2022}, the index $M$ of
$\varUpsilon_M(X,Y,s)$ matters only in the choice of the cornerstones. The
current proof is slightly more general than that of~\cite{erdos_mixing_2022},
because we not only consider $M=X+Y$, but also any other $M$ such that
$(X,Y,s)\in\mathfrak{D}_M$ (recall \cref{eq:DM}). In the path
$\varUpsilon_{M}(X,Y,s)$ incorporating $\textsc{Sweep}(G_r,C_r)$ (see
\cref{eq:defUpsilon}) choose the
cornerstone $v_r$ of the $\textsc{Sweep}(G_r,C_r)$ as follows:
\begin{align}\label{eq:cornerstone}
	\begin{split}
	&\text{Let }v_r\in V(C_r)\text{\ be the vertex which minimizes the row-sum
    in}\\
    &\qquad \left(M-A_{X\triangle \uplus_{i=1}^{r-1}C_i}\right)\Big|_{V(C_r)\times V(C_r)}\\
    &\text{and $v_r$ is lexicographically minimal with respect to this condition.}\\
	\end{split}
\end{align}

\medskip

Because $X,Y\in \mathcal{G}(d)$ for some $d$, Lemma~2.6
of~\cite{erdos_mixing_2022} applies, which claims that
$\textsc{Sweep}(X\triangle \uplus_{i=1}^{r-1}C_i,C_r)$ is a sequence of at most
$\frac12|E(C_r)|-1$ switches that connect $X\triangle \uplus_{i=1}^{r-1}C_i$ to
$X\triangle \uplus_{i=1}^{r}C_i$. Thus the total length of the switch sequence
$\varUpsilon_M(X,Y,s)$ is at most $\frac12|\nabla_{X,Y}|-1$. For any $Z\in
\varUpsilon_M(X,Y,s)$, the degree sequences of $X$, $Y$ and $Z$ are identical,
because switches preserve the degree sequence. Note that for any $j\neq r$, the sequence
$\textsc{Sweep}(X\uplus_{i=1}^{j-1}C_i,C_j)$ does not depend on the
cornerstone $v_r$.

\medskip

For any $Z\in \varUpsilon_M(X,Y,s)$, the matrix $M-A_Z$ belongs to
${\{-1,0,1,2\}}^{n\times n}$. Recall \cref{eq:DM}. If $(M-A_Z)_{vw}=2$, then
$vw$ is an edge in both $X$ and $Y$, but $vw$ is not present in $Z$ as an edge.
If, however, $(M-A_Z)_{vw}=-1$, then $vw\notin E(X),E(Y)$ and $vw\in E(Z)$. With
formulae,
\begin{align}
    \{ vw\ |\ (M-A_Z)_{vw}=+2\}&\subseteq E(X)\setminus E(Z)\setminus
    \nabla_{X,Y}\label{eq:p2}\\
    \{ vw\ |\ (M-A_Z)_{vw}=-1\}&\subseteq E(Z)\setminus E(X)\setminus
    \nabla_{X,Y}\label{eq:m1}
\end{align}
respectively. In~\cite[Lemma~2.7]{erdos_mixing_2022}, the set $R=R_Z$ is defined, and it has
cardinality at most 4. By its definition, the set of edges in $R$ is a superset
of $\left(E(X)\triangle E(Z)\right)\setminus \nabla_{X,Y}$, which is the union of the right hands sides
of \cref{eq:p2,eq:m1}. In short, every $+2$ and $-1$ entry of $M-A_Z$ is in a
position which is associated to an edge in $R$.

\medskip

We will show that $M-A_Z$ is $7$-tight. Lemma~8.2 in~\cite{erdos_mixing_2022}
is the analogue of this tightness statement, and its proof can be repeated for this
case with little to no modification. Suppose first that
every edge in $R$ is incident on $v_r$ from~\cref{eq:cornerstone}:
then~\cite[Lemma~7.1]{erdos_mixing_2022} claims that the entries in
$A_X+A_Y-A_Z$ associated to edges in $R$ consist of at most two pairs of symmetric $+2$
entries, and at most one pair of symmetric $-1$ entries. By \cref{eq:p2,eq:m1},
$M-A_Z$ also contains at most two pairs of symmetric $+2$ entries, and at most one
pair of symmetric $-1$ entries.

\medskip

Recall that $Z$ is obtained from $X\triangle
\uplus_{i=1}^{r-1}C_i$ through a series of switches that
only touch edges whose vertices are contained in $V(C_r)$. Thus the row-
and columns-sums of the submatrices
\begin{equation*}
(M-A_Z)|_{V(C_r)\times V(C_r)}\quad\text{and}\quad\left(M-A_{X\triangle
\uplus_{i=1}^{r-1}C_i}\right)\Big|_{V(C_r)\times V(C_r)}
\end{equation*}
are identical. Let $v$ and $w$ be two distinct vertices in $V(C_r)$.
\begin{itemize}
	\item If $M_{vw}=2$, then by \cref{eq:p2},  $vw\in E(X)$ and $vw\notin
		\nabla_{X,Y}$, thus
		\begin{equation*}
			{\left(M-A_{X\triangle \uplus_{i=1}^{r-1}C_i}\right)}_{vw}={(M-A_X)}_{vw}=2-1=1.
		\end{equation*}
	\item If $M_{vw}=0$, then by \cref{eq:m1}, $vw\notin E(X)$ and $vw\notin
		\nabla_{X,Y}$, thus
		\begin{equation*}
			{\left(M-A_{X\triangle \uplus_{i=1}^{r-1}C_i}\right)}_{vw}={(M-A_X)}_{vw}=0-0=0.
		\end{equation*}
	\item If $M_{vw}=1$ and $vw\in E\left(X\triangle \uplus_{i=1}^{r-1}C_i\right)$,
		then ${\left(M-A_{X\triangle \uplus_{i=1}^{r-1}C_i}\right)}_{vw}=0$.
	\item If $M_{vw}=1$ and $vw\notin E\left(X\triangle \uplus_{i=1}^{r-1}C_i\right)$,
		then ${\left(M-A_{X\triangle \uplus_{i=1}^{r-1}C_i}\right)}_{vw}=1$.
\end{itemize}
Every entry of $\left(M-A_{X\triangle
\uplus_{i=1}^{r-1}C_i}\right)\big|_{V(C_r)\times V(C_r)}$ is either a 0 or a 1,
and the diagonal is identically zero. Since $C_r$ is alternating in $X\triangle
\uplus_{i=1}^{r-1}C_i$, there is at least one $0$ entry and one $1$ entry in
every row and every column. Therefore the row- and column-sums of
$(M-A_Z)|_{V(C_r)\times V(C_r)}$ are at least $1$ and at most $|V(C_r)|-2$.
Moreover, \cref{eq:cornerstone} ensures that the row-sum corresponding to $v_r$
in ${(M-A_Z)}_{V(C_r)\times V(C_r)}$ is minimal. By \Cref{lemma:removeplus2s},
$M-A_Z$ is $5$-tight.

\medskip

We will again use~\cite[Lemma~2.7]{erdos_mixing_2022} to understand the more
detailed structure of $R_Z$. If there is an edge in $R_Z$ which is not incident
on $v_r$, then $R$ falls under case~(e) of~\cite[Lemma~2.7]{erdos_mixing_2022}.
Let $Z\triangle F$ be the next graph in the \textsc{Sweep} sequence, where $F$
is a $C_4$. By~\cite[Lemma~2.7(d)]{erdos_mixing_2022}, every edge in the set
$R_{Z\triangle F}$ is incident on $v_r$. As
previously,~\Cref{lemma:removeplus2s} implies that $M-A_{Z\triangle F}$
is $5$-tight, and thus $M-A_Z$ is $7$-tight.

\medskip

Next, we will cite 3 lemmas from~\cite{erdos_mixing_2022}. The first of these
lemmas refers to the graph $Z'=Z\triangle R$, which is defined
in~\cite[eq.~(13)]{erdos_mixing_2022}. Note that the graph $Z'$ is just a slight
perturbation of $Z$.
\begin{lemma}[adapted from Lemma~5.15 in~\cite{erdos_mixing_2022}]\label{lemma:WkSplit}
    For any $Z\in\varUpsilon(X,Y,s)$ for $s\in S_{X,Y}$, there exists
    $\pi_{Z'}\in \Pi(\nabla)$ which defines a closed Eulerian trail on
    $\nabla_{X,Y}$
    which is alternating in $Z'$ with at most 4 exceptions.
\end{lemma}

\begin{lemma}[Lemma~5.21 in~\cite{erdos_mixing_2022}]\label{lemma:Bsize}
    For a fixed number $n$ of vertices of $X$ and $Y$, the cardinality of the set of possible
    tuples $B(X,Y,Z,s)$ is $\mathcal{O}(n^{8})$, where $s\in S_{X,Y}$
    and $Z\in\varUpsilon(X,Y,s)$ are arbitrary.
\end{lemma}

\begin{lemma}[Lemma~5.22 in~\cite{erdos_mixing_2022}]\label{lemma:sigma}
	The quadruplet composed of the graphs $Z$, $\nabla$, $\pi_{Z'}$, and
    $B(X,Y,Z,s)$ uniquely determines the triplet $(X,Y,s)$.
\end{lemma}

We define $\pi_M(X,Y,s,Z)=\pi_{Z'}$. \Cref{lemma:WkSplit} implies that $\pi_{Z'}$
is alternating in $Z$ with at most $4+2|R_Z|\le 12$ exceptions. Let $B_M(X,Y,s,Z)$ be identical
with the parameter set $B(X,Y,Z,s)$ defined in~\cite{erdos_mixing_2022}.
\Cref{lemma:WkSplit,lemma:Bsize,lemma:sigma} ensure that every itemized requirement of
\Cref{def:precursor} holds, similarly to the situation
in~\cite{erdos_mixing_2022}.\\ \mbox{ }\ \mbox{ }\hfill $\qed_{\textbf{\Cref{lemma:Cthin}}}$

\medskip

Now we are at the point where \Cref{thm:switch} is reproved by the generalized
machinery: the Markov chain $\mathbb{G}(d)$ (using switches only) is rapidly mixing for any
$d$ from a $P$-stable set. By \Cref{lemma:Cthin,lemma:extension}, there exists a precursor on
$\mathfrak{R}_\mathrm{id}$ with parameter $3c$, and the theorem follows from
\Cref{rem:pstable} and \Cref{thm:precursor}.

\subsection{Stage 2: open trails}
Until now, the degree sequences of $X$ and $Y$ in
$(X,Y,s)\in\mathfrak{C}_\mathrm{thin}$ were identical, that is, $s$ was a closed
trail. In the second stage we deal
with the case when $\|\deg_X-\deg_Y\|_1=2$ (while $\|\deg_X-\deg_Y\|_\infty=1$).
The following lemma is actually a
framework for reducing the construction of the precursor on
$\mathfrak{C}_\mathrm{thin}$ to \Cref{lemma:Cthin}. Note that we do not aim to
optimize our estimate of the mixing time, we are merely interested in bounding
it polynomially. Surprisingly, to construct the precursor on
$\mathfrak{C}_\mathrm{thin}$, it is sufficient to consider only those open trails
$s$ that have odd length.

\medskip

Informally, the forthcoming \Cref{lemma:CI} states that if any open
$(X,Y)$-alternating trail of odd length can be cut up into a constant number of
segments that can be reassembled into at most two $(X,Y)$-alternating trails
that are either closed or can be closed by including $v_0v_\lambda$ or
$v_1v_{\lambda-1}$ to join the two ends (alternation is not required there),
then we can reduce the precursor construction on $\mathfrak{C}_\mathrm{thin}$ to
a precursor construction on $\mathfrak{C}_\mathrm{id}$.

\begin{lemma}\label{lemma:CI}
    Suppose there exists a precursor on $\mathfrak{C}_\mathrm{id}$ with
    parameter $c$, and let $c'$ be a fixed integer. Suppose, moreover, that for
    any $(X,Y,s)\in\mathfrak{C}_\mathrm{thin}$ where $s=v_0v_1\ldots
    v_\lambda\in \Pi(\nabla_{X,Y})$ is an open trail with $v_0<_\mathrm{lex}
    v_\lambda$ for some odd integer $\lambda$, there exist $\nabla_1,\nabla_2$
    and $s_1\in \Pi(\nabla_1),s_2\in\Pi(\nabla_2)$ (where $\nabla_2=\emptyset$
    is allowed) such that
    \begin{enumerate}[label={\rm (\arabic*)}]
        \item $\nabla_{X,Y}\setminus\{v_0v_\lambda\}\subseteq \nabla_1\cup
            \nabla_2\subseteq\left\{
             \begin{array}{ll} \nabla_{X,Y}\cup \{v_0v_\lambda\} & \text{if
                }v_1=v_{\lambda-1}\\ \nabla_{X,Y}\cup
                     \{v_0v_\lambda,v_1v_{\lambda-1}\}& \text{if
            }v_1\neq v_{\lambda-1}\end{array}\right.$,\label{assumption:nablaxy}
        \item $\nabla_{X,Y}\triangle \nabla_1\triangle \nabla_2\subseteq
            \{v_0v_\lambda\}$,\label{assumption:symmdiff}
        \item if $v_1v_{\lambda-1}\in (\nabla_1\cup\nabla_2)\setminus \nabla_{X,Y}$, then $v_0v_\lambda\in
            \nabla_{X,Y}$ and $s_1$ or $s_2$ is equal to
            $v_0v_1v_{\lambda-1}v_\lambda v_0$.\label{assumption:v1vl-1}
    \end{enumerate}
    Moreover, for both $i=1,2$:
    \begin{enumerate}[label={\rm (\arabic*)}]
        \setcounter{enumi}{3}
        \item the line graph $L(\nabla_i,s_i)$ is an even cycle (or an empty
            graph),\label{assumption:linegraph}
        \item $s_i-v_0v_\lambda-v_1v_{\lambda-1}$ is $(X,Y)$-alternating,
        \item $s_i-v_0v_\lambda$ is $(X,Y)$-alternating with 0 or 2 exceptions,
        \item $s_i-v_1v_{\lambda-1}$ is $(X,Y)$-alternating with 0 or 2 exceptions,
        \item the number of components of $L(\nabla_i,s_i)\cap L(\nabla,s)$ is
            at most $c'$.
    \end{enumerate}
    Then there exists a precursor on $\mathfrak{C}_\mathrm{thin}$ with parameter
    $3c+60c'+300$.
\end{lemma}
We are aware that such a huge parameter is nowhere near a practical
bound. We made virtually zero effort to optimize the parameter.
\begin{proof}
    Let $(X,Y,s)\in \mathfrak{C}_\mathrm{thin}$ be such that $s=v_0v_1\ldots
    v_\lambda$ and $v_0<_\mathrm{lex} v_\lambda$. We will
    now consider the case when $\lambda$ is odd. As discussed earlier, the case
    of even $\lambda$ will be handled by a reduction to the odd case. For an odd $\lambda$, we must
    have either $\deg_Y=\deg_X+\mathds{1}_{\{v_0,v_\lambda\}}$ or
    $\deg_Y=\deg_X-\mathds{1}_{\{v_0,v_\lambda\}}$, because $s$ is
    $(X,Y)$-alternating and its length $\lambda$ is odd.

    \medskip

    Let $\nabla_i$ and $s_i\in \Pi(\nabla_i)$ for $i=1,2$ be the
    set of edges and pairing function assumed to exist in the statement of this
    lemma. Let $M$ be such that $(X,Y,s)\in\mathfrak{D}_M$, and we will first define
    $\varUpsilon_M(X,Y,s)$, then we will also define $\pi_M(X,Y,s,Z)$ and $B_M(X,Y,s,Z)$ for any
    $Z\in\varUpsilon_M(X,Y,s)$.

    \medskip

    Let us modify the auxiliary matrix $M$. Recall from
    \Cref{def:DM} that if $ab\in\nabla_{X,Y}$, then $M_{ab}=1$. By
    assumption~\ref{assumption:v1vl-1} of this lemma, if
    $M_{v_1v_{\lambda-1}}\neq 1$ and $v_1v_{\lambda-1}\in
    \nabla_i$, then $v_0v_\lambda\in \nabla_{X,Y}$ and $M_{v_0v_\lambda}=1$.
    Let us define
    \begin{equation}\label{eq:Mprime}
        M'=\left\{
        \begin{array}{ll}
            M+A_{(v_0v_\lambda)} & \text{if }M_{v_0v_\lambda}=0,\\
            M-A_{(v_0v_\lambda)} & \text{if }M_{v_0v_\lambda}=2,\\
            M+A_{(v_1v_{\lambda-1})}-A_{(v_0v_1)}-A_{(v_{\lambda-1}v_\lambda)} &
            \text{if }M_{v_0v_\lambda}=1\text{ and }M_{v_1v_{\lambda-1}}=0\text{
            and }v_1\neq v_{\lambda-1},\\
            M-A_{(v_1v_{\lambda-1})}+A_{(v_0v_1)}+A_{(v_{\lambda-1}v_\lambda)} &
            \text{if }M_{v_0v_\lambda}=1\text{ and }M_{v_1v_{\lambda-1}}=2\text{
            and }v_1\neq v_{\lambda-1},\\
            M &\text{if }M_{v_0v_\lambda}=1\text{ and
            }M_{v_1v_{\lambda-1}}=1\text{ or }v_1=v_{\lambda-1},
        \end{array}\right.
    \end{equation}
    so that $M'_{v_0v_\lambda}=1$. Also, $M'_{v_1v_{\lambda-1}}=1$ if
    $v_1v_{\lambda-1}\in \nabla_1\cup
    \nabla_2$. The row-sums of $M$
    and $M'$ are equal on every vertex except possibly on $v_0$ and $v_\lambda$.

    \medskip

    By
    assumption~\ref{assumption:linegraph}, $|\nabla_i|$ is even. From
    assumptions~\ref{assumption:nablaxy}~and~\ref{assumption:symmdiff} it follows
    that any $v_{j}v_{j+1}$ is contained in either $\nabla_1$ or $\nabla_2$, but
    not both, except if $\{v_{j},v_{j+1}\}=\{v_0,v_\lambda\}$ or
    $\{v_{j},v_{j+1}\}=\{v_1,v_{\lambda-1}\}$. Therefore $\nabla_1\cap
    \nabla_2\subseteq \{ v_0v_\lambda, v_1v_{\lambda-1}\}$. Let us start to
    extend the precursor. Without loss of generality, we may assume that
    $\nabla_1\neq \emptyset$.

    \paragraph{Case A1: $\nabla_2=\emptyset$.}
    Note that $|\nabla_{X,Y}|=\lambda$ is odd. Since $|\nabla_1|$ is even and
    $\nabla_1\setminus \nabla_{X,Y}\subseteq \{v_0v_\lambda\}$, we must have
    $v_0v_{\lambda}\notin \nabla_{X,Y}\Leftrightarrow v_0v_\lambda\in \nabla_1$.
    Also, $v_1v_{\lambda-1}\in \nabla_1\Leftrightarrow v_1v_{\lambda-1}\in
    \nabla_{X,Y}$. Let us slightly change $X$ and $Y$, so that the symmetric
    difference of the modified graphs $X_1,Y_1$ is exactly $\nabla_1$:
    \begin{align*}
    X_1&=\left\{
        \begin{array}{ll}
    		X\triangle v_0v_{\lambda}, &\text{if $s_1$ is not alternating in
            $X$};\\
    		X, &\text{if $s_1$ is alternating in $X$};\\
        \end{array}
    \right.\\
    Y_1&=\left\{
        \begin{array}{ll}
    		Y, &\text{if $s_1$ is not alternating in $X$};\\
    		Y\triangle v_0v_{\lambda}, &\text{if $s_1$ is alternating in $X$};\\
        \end{array}
    \right.
    \end{align*}
    Suppose $s_1-v_0v_\lambda$ is not alternating in $X$: then
    $v_1v_{\lambda-1}\in \nabla_{X,Y}$ and the two non-alternations of
    $s_1-v_0v_\lambda$ are located at $v_1$ and $v_{\lambda-1}$. But because
    $s_1-v_0v_\lambda-v_1v_\lambda$ is alternating in $X$, we have
    $|\deg_{E(X)\cap\nabla_{X,Y}}(v_1)-\deg_{E(Y)\cap\nabla_{X,Y}}(v_1)|=2$, so $s$ cannot
    possibly be $(X,Y)$-alternating, a contradiction. It follows that $s_1$ is
    alternating in $X_1$ (and thus $Y_1$): indeed, if $s_1$ is not alternating in $X$,
    then the two exceptions are the endpoints of $v_0v_\lambda$. Therefore
    $(X_1,Y_1,s_1)\in \mathfrak{C}_\mathrm{id}\cap \mathfrak{D}_{M'}$.

    \medskip

    We extend the precursor to $(X,Y,s)$ as follows.
    \begin{align}
        \varUpsilon_{M}(X,Y,s)&=\left\{
        \begin{array}{ll}
            X\xrightarrow{\text{toggle $v_0v_\lambda$}}
            \varUpsilon_{M'}(X_1,Y_1,s_1) &\text{if $s_1$ is not alternating in $X$}\\
    		\varUpsilon_{M'}(X_1,Y_1,s_1)\xrightarrow{\text{toggle $v_0v_\lambda$}}Y &\text{if $s_1$ is alternating in $X$},
        \end{array}
        \right.\\
        \pi_M(X,Y,s,Z)&=\left\{
        \begin{array}{ll}
            s & \text{if }Z=X,Y\\
            \pi_{M'}(X_1,Y_1,s_1,Z)|_{\nabla_{X,Y}} & \text{if }Z\in
            \varUpsilon_{M'}(X_1,Y_1,s_1)\setminus \{X,Y\}
        \end{array}\right.\\
        B_M(X,Y,s,Z)&=\left\{
        \begin{array}{l}
             (0,\mathrm{true}) \hspace{8em}\text{if }Z=X\\
             \\
             (0,\mathrm{false}) \hspace{8em}\text{if }Z=Y\\
             \\
             \big(1,\lambda,v_0v_\lambda,v_{0}v_{\lambda}\in
             E(X),\pi_M(X,Y,s,Z)\triangle\pi_{M'}(X_1,Y_1,s_1,Z),B_{M'}(X_1,Y_1,s_1,Z)\big) \\
                           \hspace{11em}\text{if }Z\in
                           \varUpsilon_{M'}(X_1,Y_1,s_1)\setminus \{X,Y\}
        \end{array}\right.
    \end{align}

    Let us verify that \Cref{def:precursor} holds for the extension. The defined
    path $\varUpsilon_M(X,Y,s)$ in the Markov-graph utilizes one edge-toggle,
    while the rest of the steps are switches. When the edge-toggle occurs, the
    degree sequence of the then current graph changes from $\deg_X$
    to $\deg_Y$, because the rest of the steps do not change
    the degree sequence.

    \medskip

    If $Z=X,Y$, then $M-A_Z$ is 0-tight, because
    $(X\triangle
    Z)\setminus \nabla=(Y\triangle Z)\setminus \nabla=\emptyset$. Suppose next,
    that $Z\in \varUpsilon_{M'}(X_1,Y_1,s_1)$. If $M'=M$, then $M-A_Z$ is $c$-tight
    by induction. If $M'=M\pm A_{v_0v_\lambda}$, then note that the row-sums of
    $M'-A_{X}$ are equal to the row-sums of $M-A_Y$, and the row-sums of $M'-A_Y$
    are equal to the row-sums of $M-A_X$. The degree sequence of $Z$ is
    equal to $\deg_X$ or $\deg_Y$, so $M'-A_Z$ is $c$-tight and therefore $M-A_Z$ is $c+1$-tight.

    \medskip

    The length of $\varUpsilon_M(X,Y,s)$ is at most
    $1+c|\nabla_1|=1+c|\nabla_{X,Y}|+c$, still linear.
    The symmetric difference of $X$ and $Z$ outside $\nabla_{X,Y}$ may also include $v_0v_\lambda$, so
    the upper bound in \Cref{def:precursor}\ref{prop:cexceptions} increases by at most one.

    \medskip

    The maximum number of exceptions to alternation of $\pi_M(X,Y,s,Z)$ in $Z$
    is no more than the number of exceptions to alternation of
    $\pi_{M'}(X_1,Y_1,s_1,Z)$ in $Z$, because $v_0v_\lambda\notin \nabla_{X,Y}$. Since
    $\pi_{M'}(X_1,Y_1,s_1,Z)$ is a closed trail, even if we restrict its domain
    from $\nabla_1$ to $\nabla_{X,Y}$, it remains connected. The range of $B_M(X,Y,s,Z)$
    increases by a polynomial multiplicative factor (of at most $4n^4$, but this
    will be dwarfed by the bound in the next case).

    \medskip

    Lastly, $\varPsi$ is still well-defined. Trivially, if
    $B_M(X,Y,s,Z)=(0,\mathrm{true})$ (alternatively $(0,\mathrm{false})$), then
    $X=Z$ ($Y=Z$) and $Y=Z\triangle
    \nabla_{X,Y}$ ($X=Z\triangle \nabla_{X,Y}$). If $B_M(X,Y,s,Z)=(1,\cdots)$,
    then we can recover
    $\pi_{M'}(X_1,Y_1,s_1,Z)$ from $\pi_M(X,Y,s,Z)$ using their symmetric
    difference, and subsequently, we can recover $X_1$ and $Y_1$ via $\varPsi$,
    because we have a precursor on $\mathfrak{C}_\mathrm{id}$. From these graphs we can easily recover both $X$ and $Y$, as $B_M(X,Y,s,Z)$ describes
    whether $v_{0}v_{\lambda}$ is in $E(X)$ or not (and the same containment
    relation holds for $E(Y)$ because $v_0v_\lambda\notin \nabla_{X,Y}$).

    \paragraph{Case A2: $\nabla_1\neq\emptyset$ and $\nabla_2\neq \emptyset$.}
    \emph{Task 1: constructing $\varUpsilon_M(X,Y,s)$.}
    Obviously, $|\nabla_i|\ge 4$ for $i=1,2$ ($s_i$ is an even length closed trail), and $|\nabla_1\cap \nabla_2|\le
    2$, so $\lambda\ge 5$. The reduction is similar to the previous case,
    however, the construction of the precursor on $(X,Y,s)$ will be reduced to
    not one, but two elements of $\mathfrak{C}_\mathrm{id}$. Recall, that any
    $v_{j}v_{j+1}\neq v_0v_\lambda,v_1v_{\lambda-1}$ appears in exactly one of
    $\nabla_1$ and $\nabla_2$. If $v_1v_{\lambda-1}\in
    \nabla_{X,Y}\cup\nabla_1\cup\nabla_2$, then
    the edge $v_1v_{\lambda-1}$ appears in exactly
    two of $\nabla_{X,Y}$, $\nabla_1$, $\nabla_2$. Observe, that for any vertex
    $v\in [n]$, we have
    \begin{equation}
        \deg_X(v)-\deg_{E(X)\cap\nabla_{X,Y}}+\deg_{E(Y)\cap\nabla_{X,Y}}=\deg_Y(v).
    \end{equation}
    Thus, for any $v\neq v_0,v_\lambda$, we have
    \begin{equation}
        \deg_{E(X)\cap\nabla_{X,Y}}=\deg_{E(Y)\cap\nabla_{X,Y}}.
    \end{equation}

    \medskip

    Suppose that $s_i-v_0v_\lambda$ is not alternating in
    $X$ for $i=1$ and $i=2$. Then $s_i-v_0v_{\lambda}$ is not alternating at
    $v_1$ and $v_{\lambda-1}$, which implies that
    $v_1v_{\lambda-1}\in \nabla_1,\nabla_2$ and $v_1v_{\lambda-1}\notin
    \nabla_{X,Y}$. If, say, $v_1v_{\lambda-1}\in E(X)$, then
    $s_i(v_1,v_1v_{\lambda-1})\in E(X)$, but $s_i-v_0v_\lambda-v_1v_{\lambda-1}$ is alternating (for $i=1,2$);
    thus we have $\deg_{E(X)\cap\nabla_{X,Y}}(v_1)=\deg_{E(Y)\cap\nabla_{X,Y}}(v_1)+2$,
    so $s$ cannot possibly be $(X,Y)$-alternating, a contradiction. The case
    $v_1v_{\lambda-1}\notin E(X)$ similarly leads to a contradiction, therefore
    at least one of $s_1-v_0v_\lambda$ and $s_2-v_0v_\lambda$ must be
    alternating in $X$.

    \medskip

    By swapping $\nabla_1$ with $\nabla_2$ and $s_1$ with $s_2$, we may assume
    that $s_1-v_0v_\lambda$ is alternating in $X$. We claim that
    \begin{equation}\label{eq:s2claim}
        s_2-v_0v_\lambda\text{ is not alternating in }X\Longleftrightarrow
        v_1v_{\lambda-1}\in \nabla_1\cap \nabla_2.
    \end{equation}
    If $v_1v_{\lambda-1}\in \nabla_1,\nabla_2$, then $v_1v_{\lambda-1}\notin
    \nabla_{X,Y}$, and as before, we get a contradiction if
    $s_i-v_0v_{\lambda}$ is alternating in $X$ for both $i=1,2$, so
    $s_2-v_0v_{\lambda}$ must not alternate in $X$. If $s_2-v_0v_\lambda$ is not
    alternating in $X$, then $v_1v_{\lambda-1}\in\nabla_2$. Thus if $s_2-v_0v_\lambda$ is not
    alternating in $X$ and $v_1v_{\lambda-1}\notin \nabla_1$, then
    $v_1v_{\lambda-1}\in \nabla_{X,Y}$, and so
    $|\deg_{E(X)\cap\nabla_{X,Y}}(v_1)-\deg_{E(Y)\cap\nabla_{X,Y}}(v_1)|=2$, a
    contradiction.

    \medskip

    Let us define now 4 auxiliary graphs.
    \begin{align*}
    X_1&=\left\{
        \begin{array}{ll}
    		X\triangle v_0v_{\lambda}, &\text{if $s_1$ is not alternating in $X$}\\
    		X, &\text{if $s_1$ is alternating in $X$}\\
        \end{array}
    \right.\\
    Y_1&=X_1\triangle\nabla_1\\
    X_2&=\left\{
        \begin{array}{ll}
    		Y_1\triangle v_0v_{\lambda}, &\text{if $s_2$ is not alternating in
            $Y_1$}\\
    		Y_1, &\text{if $s_2$ is alternating in $Y_1$}\\
        \end{array}
    \right.\\
    Y_2&=X_2\triangle\nabla_2.
    \end{align*}
    By our assumptions, $s_1$ is alternating in $X_1$.
    Furthermore, from~\eqref{eq:s2claim} it follows that $s_2$ is alternating in
    $X_2$. Because $s_i$ defines an alternating trail in $X_i$, we have
    $\deg_{X_i}=\deg_{Y_i}$. Trivially, $E(X_1)\triangle E(X)\subseteq
    \{v_0v_\lambda\}$, and by the assumptions of the lemma,
    \begin{align*}
        E(Y_2)\triangle E(Y)&\subseteq \{v_0v_\lambda\}\cup \big(E(Y_1)\triangle
        \nabla_2\triangle E(Y)\big) \subseteq \{v_0v_\lambda\}\cup
        \big(E(X)\triangle \nabla_1\triangle\nabla_2\triangle E(Y)\big)\\
         E(Y_2)\triangle E(Y)&\subseteq \{v_0v_\lambda\}\cup
                            \big(\nabla_{X,Y}\triangle
                            \nabla_1\triangle\nabla_2\big)\\
         E(Y_2)\triangle E(Y)&\subseteq \{v_0v_\lambda\}
    \end{align*}
    We claim that
    \begin{equation}\label{eq:onealternates}
        \text{$s_1$ is alternating in $X$ or $s_2$ is alternating in $Y_1$ (or
        both).}
    \end{equation}
    Suppose that $s_1$ is not alternating in $X$ and $s_2$ is not alternating
    in $Y_1$. Then $X_1=X\triangle v_0v_\lambda$ and $X_2=Y_1\triangle
    v_0v_\lambda$. Because $s_1$ is not alternating in $X$, we have $v_0v_\lambda\in
    \nabla_1$. Also, because $s_2-v_1v_{\lambda-1}$ is not alternating in $Y_1=X_1\triangle
    \nabla_1=X\triangle (\nabla_1\setminus \{v_0v_\lambda\})$,
    $s_2-v_1v_{\lambda-1}$ is not alternating in $X$ either. But this implies
    that $|\deg_{X\cap \nabla_{X,Y}}(v_0)-\deg_{Y\cap \nabla_{X,Y}}(v_0)|\in
    \{2,3\}$ (depends on whether $v_0v_\lambda$ is in $\nabla_{X,Y}$ or not),
    which is a contradiction.

    \medskip

    From now on, we assume that $X_1=X$ or $X_2=Y_1$. In other words, at least
    one of the following three symmetric differences is an empty set:
    \begin{equation}\label{eq:threetriangle}
        E(X)\triangle E(X_1),E(Y_1)\triangle E(X_2),E(Y_2)\triangle
        E(Y)\subseteq \{v_0v_\lambda\}.
    \end{equation}
    If exactly one of them is an empty set, then observe that
    \begin{equation*}
        2={\|\deg_X-\deg_Y\|}_1\equiv
        {\|\deg_{X}-\deg_{X_1}\|}_1+{\|\deg_{Y_1}-\deg_{X_2}\|}_1+{\|\deg_{Y_2}-\deg_{Y}\|}_1\equiv
        2+2\pmod{4},
    \end{equation*}
    which is a contradiction. Thus there are exactly two empty sets on the left hand side of
    \cref{eq:threetriangle}. From
    $\deg_{X_i}=\deg_{Y_i}$ for $i=1,2$, it follows that
    \begin{equation}\label{eq:degXiYi}
        \deg_{X_i},\deg_{Y_i}\in \{\deg_X,\deg_Y\}\text{\ for }i=1,2.
    \end{equation}
    In other words, we have shown that
    $(X_i,Y_i,s_i)\in \mathfrak{C}_\mathrm{id}\cap \mathfrak{D}_{M'}$ for $i=1,2$, and we may
    proceed with the reduction. By~\eqref{eq:onealternates}, we have three
    cases:
    \begin{align*}
        \varUpsilon_{M}(X,Y,s)&=\left\{
        \begin{array}{ll}
            X\xrightarrow{\text{toggle $v_0v_\lambda$}}
            \varUpsilon_{M'}(X_1,Y_1,s_1)\to
            \varUpsilon_{M'}(X_2,Y_2,s_2)\to Y &\text{if $s_1$ is not
            alternating in $X$},\\
            X\to
            \varUpsilon_{M'}(X_1,Y_1,s_1)\xrightarrow{\text{toggle $v_0v_\lambda$}}
            \varUpsilon_{M'}(X_2,Y_2,s_2)\to Y &\text{if $s_2$ is not
            alternating in $Y_1$},\\
            X\to
            \varUpsilon_{M'}(X_1,Y_1,s_1)\to
            \varUpsilon_{M'}(X_2,Y_2,s_2)\xrightarrow{\text{toggle $v_0v_\lambda$}}Y &\text{otherwise},\\
        \end{array}
        \right.
    \end{align*}
    where the $\to$ signs simply represent joining two sequences (repeated graphs
    are dropped from the sequence). By the above observations about the symmetric
    differences and~\eqref{eq:degXiYi}, $\varUpsilon_M(X,Y,s)$ is indeed a path
    in the desired Markov-graph.

    \medskip

    $M'-A_Z$ is $c$-tight by the properties of the precursor on
    $\mathfrak{C}_\mathrm{id}$. Therefore $M-A_Z$ is $(c+3)$-tight.

    \medskip

    \emph{Task 2: constructing $\pi_M(X,Y,s)$.} We have to construct a connected
    $\pi_M(X,Y,s,Z)$ from the current $\pi_{M'}(X_i,Y_i,s_i,Z)$ (where $Z\in
    \varUpsilon_{M'}(X_i,Y_i,s_i)$). Notice that $L(\nabla_{X,Y},s)-\nabla_i$
    (delete $\nabla_i$ from the vertex set of the line graph) has at most $c'+1$
    components, since $L(\nabla_{X,Y},s)$ is a path. Furthermore,
    $L(\nabla_i,\pi_{M'}(X_i,Y_i,s_i))-(\nabla_i\setminus \nabla_{X,Y})$ has at
    most $2$ components (since $|\nabla_i\setminus \nabla_{X,Y}|\le 2$). For
    $Z\in \varUpsilon_{M'}(X_i,Y_i,s_i)$, let
    \begin{equation*}
        \sigma_Z=\pi_{M'}(X_i,Y_i,s_i,Z)|_{\nabla_{X,Y}}\cup
        s|_{\nabla_{X,Y}\setminus \nabla_i}.
    \end{equation*}
    The graph $L(\nabla_{X,Y},\sigma_Z)$ has at most $c'+3$ components because
    $\pi_{M'}(X_i,Y_i,s_i)|_{\nabla_{X,Y}}$ and $s|_{\nabla_{X,Y}\setminus
    \nabla_i}$ are composed of at most $2$ and $c'+1$ trails, respectively.
    Note, that
    \begin{equation*}
        \sigma_{X_i}=s_i|_{\nabla_{X,Y}}\cup s|_{\nabla_{X,Y}\setminus
        \nabla_i},
    \end{equation*}
    and thus $|\sigma_{X_i}\triangle s|\le 2(c'+3)$.

    \medskip

    We claim that there exists $\sigma'_Z\in \Pi(\nabla_{X,Y})$ such that
    $\sigma'_Z\supseteq \sigma_Z$ (extends $\sigma_Z$) and
    $|\sigma'_Z\triangle \sigma_Z|\le 2(c'+3)$. Let $U_x=\{xy\in
    \nabla_{X,Y}\ |\ (x,xy)\notin \mathrm{dom}(\sigma_Z)\}$ be the set of
    unpaired edges incident to $x$. In total, we have $\sum_{x\in
    [n]}|U_x|\le 2(c'+3)$. It is sufficient now to define
     $\sigma'_Z(x,\bullet)$ on $U_x$ for every $x\in [n]$. To do so, observe that:
    \begin{equation*}
        |U_x|=\deg_{\nabla_{X,Y}}(x)-|\{
            (x,xy)\in\mathrm{dom}(\pi_{M'}(X_i,Y_i,s_i)|_{\nabla_{X,Y}})\}|-|\{
    (x,xy)\in\mathrm{dom}(s|_{\nabla_{X,Y}\setminus \nabla_i})\}|.
    \end{equation*}
    Then the parity of $|U_x|$ satisfies:
    \begin{align*}
             |U_x|&\equiv \deg_{\nabla_{X,Y}}(x)+|\{
    (x,xy)\in\mathrm{dom}(s|_{\nabla_{X,Y}\setminus \nabla_i})\}|\pmod{2}\\
             |U_x|&\equiv\deg_{\nabla_{X,Y}}(x)+|\{ xy\in {\nabla_{X,Y}\setminus \nabla_i}\ |\
        s(x,xy)\in {\nabla_{X,Y}\setminus \nabla_i}\}|\pmod{2}\\
                 |U_x|&\equiv \deg_{\nabla_{X,Y}}(x)+I_{x=v_0}\cdot I_{v_0v_1\notin
                 \nabla_i}+I_{x=v_\lambda}\cdot I_{v_{\lambda-1}v_\lambda\notin \nabla_i}\pmod{2}
     \end{align*}
     From the last congruence it follows that $|U_x|$ is even for $x\neq
     v_0,v_\lambda$, so we may choose $\sigma'_Z(x,\bullet)$ such that it pairs
     the edges in $U_x$. If $v_0v_1\notin\nabla_i$, then $|U_{v_0}|$ is even, and
     we may choose $\sigma'_Z(v_0,\bullet)$ such that it pairs the edges in $U_x$
     (note that $\sigma'_Z(v_0,v_0v_1)=\sigma_Z(v_0,v_0v_1)=v_0v_1$). If
     $v_0v_1\in \nabla_i$, then $|U_{v_0}|$ is odd and by definition
     $\pi_{M'}(X_i,Y_i,s_i,Z)$ cannot map $(v_0,v_0v_1)$ to $v_0v_1$; thus
     $\sigma'_Z(v_0,\bullet)$ can pair all edges of $U_{v_0}$ except one, which
     $\sigma'_Z(v_0,\bullet)$ will map to itself. Define
     $\sigma'_Z(v_\lambda,\bullet)$ on $U_{v_\lambda}$ analogously. In any case,
     $L(\nabla_{X,Y},\sigma'_Z)$ is composed of a path and a certain number of
     cycles, in total still no more than $c'+3$ components.

    \medskip

    Furthermore, we claim that there exists $\pi_Z\in \Pi(\nabla_{X,Y})$ such that
    $|\pi_Z\triangle \sigma'_Z|\le 4(c'+3)$ and $L(\nabla_{X,Y},\pi_Z)$ is
    connected. The pairing function $\sigma'_Z$ defines one open trail and at most
    $2(c'+3)-1$ closed trails in $\nabla_{X,Y}$, and these trails partition the
    edge set of the connected trail $s$. Any closed trail intersecting the open
    trail can be incorporated into the open trail by changing the pairing
    function such that the symmetric difference increases by 4. 
    \medskip

    Let the pairing function associated to $Z$ be
    \begin{align*}
    \pi_M(X,Y,s,Z)=\pi_Z.
    \end{align*}
    We know that $\sigma_Z$ alternates with at most $3c$ exceptions in $Z$, since
    $|(E(X_i)\triangle E(Z))\setminus \nabla_{i}|\le c$ and
    $\pi_{M'}(X_i,Y_i,s_i)$ alternates in $Z$ with at most $c$ exceptions. Since $|\pi_Z\triangle \sigma_Z|\le 6(c'+3)$, we get that
    $\pi_Z$ alternates in $Z$ with at most $9(c'+3)$ exceptions.

    \medskip

    \emph{Task 3: constructing $B_M(X,Y,s)$.} Let us identify the ends of intervals of $\nabla_i$ edges in a pairing function
    $\vartheta$:
    \begin{align*}
        T_Z(\vartheta)&=\{ (x,xy)\ |\ xy\in \nabla_i\text{ and }
         \left((x,xy)\notin\mathrm{dom}(\vartheta)\text{ or }\vartheta(x,xy)\notin
     \nabla_i\right)\}\\
         C_Z(\vartheta)&=\{ \min_\mathrm{lex}V(L)\ |\ L\text{ is a component in
         }L(\nabla_i\cap \nabla_{X,Y},\vartheta)\},
    \end{align*}
    where $\displaystyle\min_\mathrm{lex}V(L)$ stands for the lexicographically minimal edge
    in $V(L)$. Retracing the steps by which $\pi_Z$ is obtained, we have
    \begin{align}
        |T_Z(\pi_Z)|&\le |T_Z(\sigma_Z)|+6(c'+3)\le
        \big|T_Z(\pi_{M'}(X_i,Y_i,s_i)|_{\nabla_{X,Y}}\big|+6(c'+3)\le
        8+6(c'+3),\\
        |C_Z(\pi_Z)|&\le |C_Z(\sigma_Z)|+6(c'+3)\le
        \big|C_Z(\pi_{M'}(X_i,Y_i,s_i)|_{\nabla_{X,Y}}\big|+6(c'+3)\le
        2+6(c'+3).
    \end{align}
    Let
    \begin{align*}
    B_M(X,Y,s,Z)&=\left\{
    \begin{array}{ll}
         (0,Z\equiv X) \hspace{6em} \text{if }Z=X,Y\\ \\
         \big(2,\lambda,v_0v_\lambda,v_{0}v_{\lambda}\in
         E(X),v_1,v_{\lambda-1},(s_i|_{\nabla_{X,Y}}\cup s|_{\nabla_{X,Y}\setminus
     \nabla_i})\triangle s,T_Z(\pi_Z),C_Z(\pi_Z),\\
         \hspace{1em}i,\pi_Z\triangle \sigma_Z,\pi_M(X,Y,s,Z)|_{\nabla_i}\triangle\pi_{M'}(X_i,Y_i,s_i,Z),B_{M'}(X_i,Y_i,s_i,Z)\big)
 \\
\hspace{11em}\text{if }Z\in \varUpsilon_{M'}(X_i,Y_i,s_i)\setminus \{X,Y\}\\
    \end{array}\right.
    \end{align*}
    Every set listed in $B_M(X,Y,s,Z)$ has at most a constant size, so
    the size of the range of $B_M$ increases by a polynomial factor of $n$ (of
    at most $n^{60c'+240}$).
    It remains to show that $\varPsi$ is still well-defined. This is trivial if
    $Z=X,Y$. Suppose from now on, that $Z\in \varUpsilon_{M'}(X_i,Y_i,s_i)$.
    Since $L(\nabla_{X,Y},\pi_Z)$ is composed of paths and cycles, $T_Z$
    determines the ends of intervals of consecutive $\nabla_i$ edges in the trails determined by $\pi_Z$,
    and $C_Z$ determines those $L(\nabla_{X,Y},\pi_Z)$ components whose
    vertex set is a subset of $\nabla_i$. Therefore $\nabla_{X,Y},T_Z,C_Z$ and
    $\pi_Z$ determine $\nabla_i$. Thus $\pi_Z|_{\nabla_i}=\pi_M(X,Y,s,Z)|_{\nabla_i}$ can be
    determined, and in turn $\pi_{M'}(X_i,Y_i,s_i,Z)$ can be reconstructed too.
    Because we have a precursor on $\mathfrak{C}_\mathrm{id}$, we get
    \begin{equation*}
         (X_i,Y_i,s_i)=\varPsi(Z,\nabla_i,\pi_{M'}(X_i,Y_i,s_i,Z),B_{M'}(X_i,Y_i,s_i,Z)).
    \end{equation*}
    Notice, that $X-v_0v_\lambda=X_1-v_0v_\lambda$ and
    $Y-v_0v_\lambda=Y_2-v_0v_\lambda$.
    Since $v_0v_\lambda\in E(Y)$ if and only if $v_0v_\lambda\in
    E(X)\triangle \nabla_{X,Y}$, both $X$ and $Y$ are determined by $(X_i,Y_i)$.
    Furthermore, $\sigma_Z|_{\nabla_{X,Y}\setminus
    \nabla_i}=s|_{\nabla_{X,Y}\setminus \nabla_i}$ is already determined, and
    together with $s_i|_{\nabla_{X,Y}}$ and the auxiliary parameters, they
    determine $s$.

    \medskip

    We have now defined the precursor on any
    $(X,Y,s)\in\mathfrak{C}_\mathrm{thin}$ where $s$ is an
    open trail of odd length. Suppose from now on that
    $(X,Y,s)\in\mathfrak{C}_\mathrm{thin}$ where $s$ is an open trail of even
    length.

    \paragraph{Case B: $s=v_0v_1\ldots v_{\lambda-1}v_\lambda$ is an open trail of
    even length and $v_0=v_{\lambda-1}$.}
    We will perform exactly one
    hinge-flip $\{v_{\lambda-2}v_0,v_{\lambda-2}v_\lambda\}$. This case is very
    similar to when $s$ is an open trail of odd length and $\nabla_2=\emptyset$,
    so we will give the construction, but checking the precursor properties is
    left to the diligent reader. Let
    \begin{align*}
        s_1&=v_0v_1\ldots v_{\lambda-2}v_\lambda v_0\\
        \nabla_1&=\nabla_{X,Y}\triangle
        \{v_{\lambda-2}v_0,v_{\lambda-2}v_\lambda\}\\
        M'&=\left\{
            \begin{array}{ll}
                M+A_{(v_{\lambda-2}v_\lambda)} & \text{if
                    }M_{v_{\lambda-2}v_\lambda}=0\\
                M-A_{(v_{\lambda-2}v_\lambda)} & \text{if
                    }M_{v_{\lambda-2}v_\lambda}=2\\
                M & \text{if
                    }M_{v_{\lambda-2}v_\lambda}=1
            \end{array}
            \right.\\
    X_1&=\left\{
        \begin{array}{ll}
    		X\triangle \{v_{\lambda-2}v_0,v_{\lambda-2}v_\lambda\}, &\text{if $s_1$ is not alternating in $X$}\\
    		X, &\text{if $s_1$ is alternating in $X$}\\
        \end{array}
    \right.\\
    Y_1&=\left\{
        \begin{array}{ll}
    		Y, &\text{if $s_1$ is not alternating in $X$}\\
    		Y\triangle \{v_{\lambda-2}v_0,v_{\lambda-2}v_\lambda\}, &\text{if $s_1$ is alternating in $X$}\\
        \end{array}
    \right.\\
        \varUpsilon_{M}(X,Y,s)&=\left\{
        \begin{array}{ll}
            X\xrightarrow{\text{hinge-flip}}
            \varUpsilon_{M'}(X_1,Y_1,s_1) &\text{if $s_1$ is not alternating in $X$}\\
    		\varUpsilon_{M'}(X_1,Y_1,s_1)\xrightarrow{\text{hinge-flip}}Y &\text{if $s_1$ is alternating in $X$}
        \end{array}
        \right.
    \end{align*}
    We define $\pi_M(X,Y,s,Z)$ simply by replacing the
    $(v_{\lambda-2},v_{\lambda-2}v_{\lambda})$ with
    $(v_{\lambda-2},v_{\lambda-2}v_0)$ in the pairing
    $\varUpsilon_{M'}(X_1,Y_1,s_1)$, remove
    $(v_{\lambda},v_{\lambda-2}v_{\lambda})$ from the pairing (and create the
    self-paired edges at $v_0$ and $v_{\lambda}$). Defining a suitable $B_M(X,Y,s,Z)$ is
    straightforward and it is also left to the reader.

    \paragraph{Case C: $s=v_0v_1\ldots v_{\lambda-1}v_\lambda$ is an open trail of
    even length and $v_0\neq v_{\lambda-1}$.}
    Let $s'=s-v_{\lambda-1}v_\lambda$ and observe that we have already defined
    $\varUpsilon_M(X,Y\triangle v_{\lambda-1}v_\lambda,s')$ in the previous
    subsection, since $L(\nabla-v_{\lambda-1}v_\lambda,s')$ is a path of odd
    length.

    \medskip

    However, choosing $\varUpsilon_M(X,Y,s)=\varUpsilon_M(X,Y\triangle
    v_{\lambda-1}v_\lambda,s')\to Y$ violates the precursor property because the
    degree at $v_{\lambda-1}$ may become too small or too large when the
    edge-toggle is performed on $v_0v_{\lambda-1}$ (the rest of the steps are
    switches). Fortunately, this is very easy to fix: simply replace the
    edge-toggle on $v_0v_{\lambda-1}$ in the previous definitions of
    $\varUpsilon_M$ with the hinge-flip between $v_0v_{\lambda-1}$ and
    $v_{\lambda-1}v_\lambda$ to obtain $\varUpsilon_M(X,Y\triangle
    {v_\lambda-1}v_\lambda,s')$. Since every other
    step in $\varUpsilon_M(X,Y\triangle v_{\lambda-1}v_\lambda,s')$ is a switch,
    this ensures that for any $Z\in \varUpsilon_M(X,Y,s)$ we have $\deg_Z\in
    \{\deg_X,\deg_Y\}$.

    \medskip

    We also need to define $\pi_M$ and $B_M$. Since the odd length case already
    describes a trail from $v_0$ to $v_{\lambda-1}$, we can join
    $v_{\lambda-1}v_\lambda$ to the edge ending the trail at $v_{\lambda-1}$ to
    obtain a suitable $\pi_M(X,Y,s,Z)$ (in the derived bounds, this
    essentially increases $c'$ by 1). Furthermore, we also need to store in
    $B_M(X,Y,s,Z)$ that the identify of $v_{\lambda-1}$ and $v_\lambda$. Since
    for any $Z\in\varUpsilon_M(X,Y,s)$. As a result, the range of $B_M(X,Y,s,Z)$
    increases by a polynomial factor (also note that the parameter of the
    precursor has to be increased by a constant to accommodate
    $v_{\lambda-1}v_\lambda$).

    \medskip

    The well-definedness of $\varPsi$ follows, because the constant number of
    differences compared to the previous case are all stored/noted in $B_M(X,Y,s,Z)$.

    \medskip

    One can say that this is proof is not very detailed, but we think it is not
    worth describing the details, because it would be an almost verbatim
    repetition of the first two cases.
\end{proof}

\section{Proof of \Cref{thm:main}.}

Let $\mathcal{I}$ be a set of weakly $P$-stable thin degree sequence intervals.
By \Cref{lemma:Cthin}, there exists a precursor with parameter $c=12$ on
$\mathfrak{C}_\mathrm{id}$. We want to apply \Cref{lemma:CI} to prove that
there exists a precursor on $\mathfrak{C}_\mathrm{thin}$ with some fixed
parameter. Showing this, \Cref{thm:main} follows: the precursor can be
extended to $\mathfrak{R}_\mathrm{thin}$ by \Cref{lemma:extension}, which is
sufficient for proving rapid maxing of $\mathbb{G}(\ell,u)$ on every
$(\ell,u)\in\mathcal{I}$ by \Cref{thm:precursor}. Suppose $(X,Y,s)\in
\mathfrak{C}_\mathrm{thin}$. If $s$ is a closed trail, then
$(X,Y,s)\in\mathfrak{C}_\mathrm{id}$, on which we have already defined a
precursor.

\medskip

Suppose from now on that $s=v_0v_1\ldots v_{\lambda}$ is an open trail of odd
length (possibly 1). By the KD-lemma (\Cref{lemma:decomp}), $v_0\neq v_\lambda$.
To apply \Cref{lemma:CI}, it is enough to define $s_1$ and $s_2$, since their
domains determine $\nabla_1,\nabla_2$. The premises of \Cref{lemma:CI} are
elementary and trivial to check once $s_1\in \Pi(\nabla_1)$ and $s_2\in
\Pi(\nabla_2)$ are given. We will finish the proof by a complete case analysis,
where we provide a suitable $s_1$ and $s_2$ for each case.

\medskip

We will prove that \Cref{lemma:CI} holds for $\mathfrak{C}_\mathrm{thin}$ with
$c'=2$. We will distinguish between 8 main cases, 3 of these have 2 subcases.
The cases will be distinguished based on the relationship between
$v_0,v_1,v_{\lambda-1},v_{\lambda}$ and $s$. Recall that $v_0\neq v_\lambda$. On
the corresponding figures, by exchanging $X$ and $Y$, we may suppose that
$v_0v_1\in E(X)$. Thus the edges of $X$ are drawn with solid lines, edges of $Y$
with dashed lines, and unknown is dotted. Those pairs that are contained in
$\nabla_{X,Y}$ are joined by thick solid or dashed lines. The similarly thick
dash-dotted lines represent $(X,Y)$-alternating segments of the trail $s$.
Recall, that a trail may visit a vertex multiple times, but it can only traverse
an edge at most once.

\paragraph{Case 1.} First we assume that $v_0 v_\lambda\notin \nabla_{X,Y}$.

\begin{minipage}[c]{.5\textwidth}
\begin{align*}
    s_1 &=v_0v_1\ldots v_{\lambda-1}v_{\lambda}v_0,\\
    s_2 &=\emptyset.
\end{align*}
\end{minipage}%
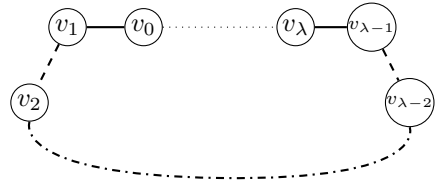
\begin{minipage}[c]{.5\textwidth}
    \centering
    \begin{tikzpicture}
        \node[draw, circle, inner sep=0pt, minimum size=14pt] (v0) at (0,-1) {$v_0$};
        \node[draw, circle, inner sep=0pt, minimum size=14pt] (v1) at (-1,-1) {$v_1$};
        \node[draw, circle, inner sep=0pt, minimum size=14pt] (v2) at (-1.5,-2) {$v_2$};
        \node[draw, circle, inner sep=0pt, minimum size=14pt] (vl0) at (2,-1) {$v_\lambda$};
        \node[draw, circle, inner sep=0pt, minimum size=14pt] (vl1) at (3,-1) {\scriptsize $v_{\lambda-1}$};
        \node[draw, circle, inner sep=0pt, minimum size=14pt] (vl2) at (3.5,-2) {\scriptsize $v_{\lambda-2}$};

        \draw (v0) edge[dotted] (vl0);
        \draw[thick] (v0)--(v1) (v1) edge[dashed] (v2) (vl0)--(vl1) (vl1)
            edge[dashed] (vl2);

        \draw[thick,dashdotted] plot [smooth, tension=2] coordinates { (v2.south)
            (1,-3.0) (vl2.south) };
    \end{tikzpicture}
\end{minipage}

\medskip

\textbf{From now on,} we assume that $v_0 v_\lambda\in \nabla_{X,Y}$. In other words, the
open trail $s\in \Pi(\nabla_{X,Y})$ traverses $v_0v_\lambda$, that is, there exists
$2\le j\le \lambda-2$ such that $\{v_0,v_\lambda\}=\{v_j,v_{j+1}\}$.

\paragraph{Case 2.} We assume in this case that $j$ is even.

\paragraph{Case 2a.} If $v_0=v_j$ and $v_\lambda=v_{j+1}$, then let

\begin{minipage}[c]{.5\textwidth}
\begin{align*}
	s_1&=v_0v_1\ldots v_{j-1}v_j,\\
    s_2&=v_{j+1}v_{j+2}\ldots v_{\lambda-1}v_{\lambda}.
\end{align*}
\end{minipage}%
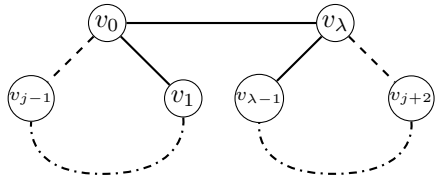
\begin{minipage}[c]{.5\textwidth}
    \centering
    \begin{tikzpicture}
        \node[draw, circle, inner sep=0pt, minimum size=14pt] (v0) at (0,0) {$v_0$};
        \node[draw, circle, inner sep=0pt, minimum size=14pt] (v1) at (1,-1) {$v_1$};
        \node[draw, circle, inner sep=0pt, minimum size=14pt] (vj1) at (-1,-1) {\scriptsize $v_{j-1}$};
        \node[draw, circle, inner sep=0pt, minimum size=14pt] (vj2) at (4,-1) {\scriptsize $v_{j+2}$};
        \node[draw, circle, inner sep=0pt, minimum size=14pt] (vl0) at (3,0) {$v_\lambda$};
        \node[draw, circle, inner sep=0pt, minimum size=14pt] (vl1) at (2,-1) {\scriptsize $v_{\lambda-1}$};

        \draw[thick] (v0)--(v1) (v0) edge[dashed] (vj1) (vl0)--(vl1) (vl0) edge[dashed] (vj2) (v0)--(vl0);
        \draw[thick,dashdotted] plot [smooth, tension=2] coordinates { (v1.south) (0,-2) (vj1.south) };
        \draw[thick,dashdotted] plot [smooth, tension=2] coordinates { (vj2.south) (3,-2) (vl1.south) };
    \end{tikzpicture}
\end{minipage}

\paragraph{Case 2b.} If $v_0=v_{j+1}$ and $v_\lambda=v_{j}$, then let

\begin{minipage}[c]{.5\textwidth}
\begin{align*}
	s_1&=v_0v_1\ldots v_{j-1}v_{j}v_{\lambda-1}v_{\lambda-2}\ldots v_{j+2}v_{j+1},\\
    s_2&=\emptyset.
\end{align*}
\end{minipage}%
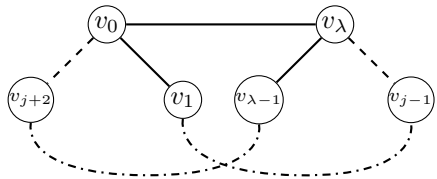
\begin{minipage}[c]{.5\textwidth}
    \centering
    \begin{tikzpicture}
        \node[draw, circle, inner sep=0pt, minimum size=14pt] (v0) at (0,0) {$v_0$};
        \node[draw, circle, inner sep=0pt, minimum size=14pt] (v1) at (1,-1) {$v_1$};
        \node[draw, circle, inner sep=0pt, minimum size=14pt] (vj1) at (-1,-1) {\scriptsize $v_{j+2}$};
        \node[draw, circle, inner sep=0pt, minimum size=14pt] (vj2) at (4,-1) {\scriptsize $v_{j-1}$};
        \node[draw, circle, inner sep=0pt, minimum size=14pt] (vl0) at (3,0) {$v_\lambda$};
        \node[draw, circle, inner sep=0pt, minimum size=14pt] (vl1) at (2,-1) {\scriptsize $v_{\lambda-1}$};

        \draw[thick] (v0)--(v1) (v0) edge[dashed] (vj1) (vl0)--(vl1) (vl0) edge[dashed] (vj2) (v0)--(vl0);
        \draw[thick,dashdotted] plot [smooth, tension=2] coordinates {
            (v1.south) (2.5,-2) (vj2.south) };
        \draw[thick,dashdotted] plot [smooth, tension=2] coordinates { (vl1.south) (0.5,-2) (vj1.south) };
    \end{tikzpicture}
\end{minipage}

\textbf{From now on,} we assume that $j$ is odd.

\paragraph{Case 3.} If $v_{j}=v_\lambda$ and $v_{j+1}=v_0$, then let

\begin{minipage}[c]{.5\textwidth}
\begin{align*}
	s_1&=v_0v_1\ldots v_j v_{j+1},\\
	s_2&=v_{j}v_{j+1}\ldots v_{\lambda-1}v_{\lambda}.
\end{align*}
\end{minipage}%
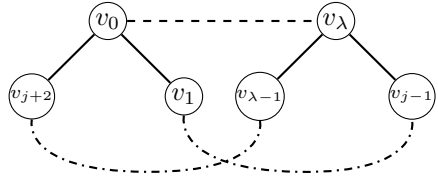
\begin{minipage}[c]{.5\textwidth}
    \centering
    \begin{tikzpicture}
        \node[draw, circle, inner sep=0pt, minimum size=14pt] (v0) at (0,0) {$v_0$};
        \node[draw, circle, inner sep=0pt, minimum size=14pt] (v1) at (1,-1) {$v_1$};
        \node[draw, circle, inner sep=0pt, minimum size=14pt] (vj1) at (-1,-1) {\scriptsize $v_{j+2}$};
        \node[draw, circle, inner sep=0pt, minimum size=14pt] (vj2) at (4,-1) {\scriptsize $v_{j-1}$};
        \node[draw, circle, inner sep=0pt, minimum size=14pt] (vl0) at (3,0) {$v_\lambda$};
        \node[draw, circle, inner sep=0pt, minimum size=14pt] (vl1) at (2,-1) {\scriptsize $v_{\lambda-1}$};

        \draw[thick] (v0)--(v1) (v0) -- (vj1) (vl0)--(vl1) (vl0) -- (vj2) (v0) edge[dashed] (vl0);
        \draw[thick,dashdotted] plot [smooth, tension=2] coordinates {
            (v1.south) (2.5,-2) (vj2.south) };
        \draw[thick,dashdotted] plot [smooth, tension=2] coordinates { (vl1.south) (0.5,-2) (vj1.south) };
    \end{tikzpicture}
\end{minipage}

\medskip

\textbf{From now on,} we assume that $v_{j}=v_0$ and
$v_{j+1}=v_\lambda$.
\paragraph{Case 4.} If $v_1=v_{\lambda-1}$, then let

\begin{minipage}[c]{.5\textwidth}
\begin{align*}
	s_1&=v_{1}v_{2}\ldots v_{j}v_{j+1}v_{1},\\
	s_2&=v_j v_{j+1}\ldots v_{\lambda-2}v_{\lambda-1}v_0.
\end{align*}
\end{minipage}%
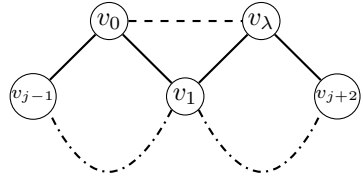
\begin{minipage}[c]{.5\textwidth}
    \centering
    \begin{tikzpicture}
        \node[draw, circle, inner sep=0pt, minimum size=14pt] (v0) at (0,0) {$v_0$};
        \node[draw, circle, inner sep=0pt, minimum size=14pt] (v1) at (1,-1) {$v_1$};
        \node[draw, circle, inner sep=0pt, minimum size=14pt] (vj1) at (-1,-1) {\scriptsize $v_{j-1}$};
        \node[draw, circle, inner sep=0pt, minimum size=14pt] (vj2) at (3,-1) {\scriptsize $v_{j+2}$};
        \node[draw, circle, inner sep=0pt, minimum size=14pt] (vl0) at (2,0) {$v_\lambda$};

        \draw[thick] (v0)--(v1) (v0) -- (vj1) (vl0)--(v1) (vl0) -- (vj2) (v0) edge[dashed] (vl0);
        \draw[thick,dashdotted] plot [smooth, tension=1]
            coordinates { (v1.south west) (0,-2) (vj1.south east) };
        \draw[thick,dashdotted] plot [smooth, tension=1]
            coordinates { (v1.south east) (2,-2) (vj2.south west) };
    \end{tikzpicture}
\end{minipage}

\medskip

\textbf{From now on,} we assume that $v_1\neq v_{\lambda-1}$.

\paragraph{Case 5.} If $v_1v_{\lambda-1}\notin \nabla_{X,Y}$.

\begin{minipage}[c]{.5\textwidth}
\begin{align*}
	s_1&=v_{j}v_{j-1}\ldots v_2v_1v_{\lambda-1}v_{\lambda-2}\ldots v_{j+1} v_j,\\
	s_2&=v_0v_1v_{\lambda-1}v_{\lambda}v_0.
\end{align*}
\end{minipage}%
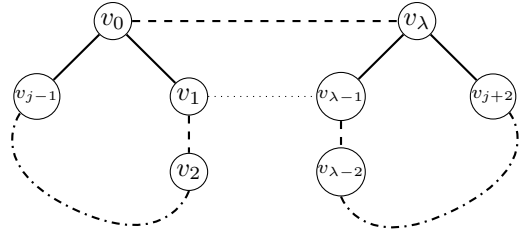
\begin{minipage}[c]{.5\textwidth}
    \centering
    \begin{tikzpicture}
        \node[draw, circle, inner sep=0pt, minimum size=14pt] (v0) at (0,0) {$v_0$};
        \node[draw, circle, inner sep=0pt, minimum size=14pt] (v1) at (1,-1) {$v_1$};
        \node[draw, circle, inner sep=0pt, minimum size=14pt] (v2) at (1,-2) {$v_2$};
        \node[draw, circle, inner sep=0pt, minimum size=14pt] (vj1) at (-1,-1) {\scriptsize $v_{j-1}$};
        \node[draw, circle, inner sep=0pt, minimum size=14pt] (vj2) at (5,-1) {\scriptsize $v_{j+2}$};
        \node[draw, circle, inner sep=0pt, minimum size=14pt] (vl0) at (4,0) {$v_\lambda$};
        \node[draw, circle, inner sep=0pt, minimum size=14pt] (vl1) at (3,-1) {\scriptsize $v_{\lambda-1}$};
        \node[draw, circle, inner sep=0pt, minimum size=14pt] (vl2) at (3,-2) {\scriptsize $v_{\lambda-2}$};

        \draw[dotted] (v1) -- (vl1);
        \draw[thick] (v0)--(v1) (v0) -- (vj1) (vl0)--(vl1) (vl0) -- (vj2) (v0)
            edge[dashed] (vl0) (v1) edge[dashed] (v2) (vl1) edge[dashed] (vl2);
        \draw[thick,dashdotted] plot [smooth, tension=2] coordinates { (v2.south) (-0.5,-2.5) (vj1.south west) };
        \draw[thick,dashdotted] plot [smooth, tension=2] coordinates { (vl2.south) (4.5,-2.5) (vj2.south east) };
    \end{tikzpicture}
\end{minipage}

\medskip

\textbf{From now on,} we assume that $v_1 v_{\lambda-1}\in \nabla_{X,Y}$. In other words, the
open trail $s\in \Pi(\nabla_{X,Y})$ traverses $v_1v_{\lambda-1}$, that is, there exists
$1\le k\le \lambda-1$ such that $\{v_1,v_{\lambda-1}\}=\{v_k,v_{k+1}\}$.

\medskip

First, we assume that $\boldmath k<j$; the case $k>j$ will follow
easily by symmetry.

\paragraph{Case 6.} Suppose that $k$ is even.

\paragraph{Case 6a.} If $v_k=v_{1}$ and $v_{k+1}=v_{\lambda-1}$, then let

\begin{minipage}[c]{.5\textwidth}
\begin{align*}
	s_1&=v_{k+1}v_{k+2}\ldots v_{j-1}v_{j}v_{j+1}v_{\lambda-1},\\
	s_2&=v_0v_1\ldots v_{k-1}v_{k}v_{\lambda-1}v_{\lambda-2}\ldots v_{j+2}v_{j+1}v_j
\end{align*}
\end{minipage}%
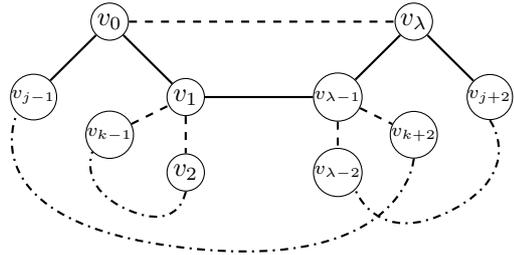
\begin{minipage}[c]{.5\textwidth}
    \centering
    \begin{tikzpicture}
        \node[draw, circle, inner sep=0pt, minimum size=14pt] (v0) at (0,0) {$v_0$};
        \node[draw, circle, inner sep=0pt, minimum size=14pt] (v1) at (1,-1) {$v_1$};
        \node[draw, circle, inner sep=0pt, minimum size=14pt] (v2) at (1,-2) {$v_2$};
        \node[draw, circle, inner sep=0pt, minimum size=14pt] (vj1) at (-1,-1) {\scriptsize $v_{j-1}$};
        \node[draw, circle, inner sep=0pt, minimum size=14pt] (vj2) at (5,-1) {\scriptsize $v_{j+2}$};
        \node[draw, circle, inner sep=0pt, minimum size=14pt] (vl0) at (4,0) {$v_\lambda$};
        \node[draw, circle, inner sep=0pt, minimum size=14pt] (vl1) at (3,-1) {\scriptsize $v_{\lambda-1}$};
        \node[draw, circle, inner sep=0pt, minimum size=14pt] (vl2) at (3,-2) {\scriptsize $v_{\lambda-2}$};

        \node[draw, circle, inner sep=0pt, minimum size=14pt] (vk1) at (0,-1.5)
            {\scriptsize $v_{k-1}$};
        \node[draw, circle, inner sep=0pt, minimum size=14pt] (vk2) at (4,-1.5)
            {\scriptsize $v_{k+2}$};

        \draw[thick] (v0)--(v1) (v0) -- (vj1) (vl0)--(vl1) (vl0) -- (vj2) (v0)
            edge[dashed] (vl0) (v1) edge[dashed] (v2) (vl1) edge[dashed] (vl2)
            (vk1) edge[dashed] (v1) (v1) -- (vl1) (vl1) edge[dashed] (vk2);
        \draw[thick,dashdotted] plot [smooth, tension=2] coordinates {
            (v2.south) (0.25,-2.5) (vk1.south west) };
        \draw[thick,dashdotted] plot [smooth, tension=2] coordinates {
            (vk2.south) (1,-3) (vj1.south west) };
        \draw[thick,dashdotted] plot [smooth, tension=2] coordinates {
            (vj2.south) (4.5,-2.5) (vl2.south east) };
    \end{tikzpicture}
\end{minipage}

\paragraph{Case 6b.} If $v_{k+1}=v_{1}$ and $v_{k}=v_{\lambda-1}$, then let 

\begin{minipage}[c]{.5\textwidth}
\begin{align*}
	s_1&=v_0v_{1}\ldots v_{k-1}v_{k}v_{\lambda}v_0,\\
    s_2&=v_{k}v_{k+1}v_{k+2}\ldots v_{\lambda-2}v_{\lambda-1}.
\end{align*}
\end{minipage}%
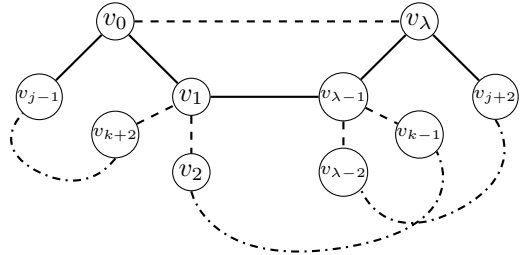
\begin{minipage}[c]{.5\textwidth}
    \centering
    \begin{tikzpicture}
        \node[draw, circle, inner sep=0pt, minimum size=14pt] (v0) at (0,0) {$v_0$};
        \node[draw, circle, inner sep=0pt, minimum size=14pt] (v1) at (1,-1) {$v_1$};
        \node[draw, circle, inner sep=0pt, minimum size=14pt] (v2) at (1,-2) {$v_2$};
        \node[draw, circle, inner sep=0pt, minimum size=14pt] (vj1) at (-1,-1) {\scriptsize $v_{j-1}$};
        \node[draw, circle, inner sep=0pt, minimum size=14pt] (vj2) at (5,-1) {\scriptsize $v_{j+2}$};
        \node[draw, circle, inner sep=0pt, minimum size=14pt] (vl0) at (4,0) {$v_\lambda$};
        \node[draw, circle, inner sep=0pt, minimum size=14pt] (vl1) at (3,-1) {\scriptsize $v_{\lambda-1}$};
        \node[draw, circle, inner sep=0pt, minimum size=14pt] (vl2) at (3,-2) {\scriptsize $v_{\lambda-2}$};

        \node[draw, circle, inner sep=0pt, minimum size=14pt] (vk1) at (0,-1.5)
            {\scriptsize $v_{k+2}$};
        \node[draw, circle, inner sep=0pt, minimum size=14pt] (vk2) at (4,-1.5)
            {\scriptsize $v_{k-1}$};

        \draw[thick] (v0)--(v1) (v0) -- (vj1) (vl0)--(vl1) (vl0) -- (vj2) (v0)
            edge[dashed] (vl0) (v1) edge[dashed] (v2) (vl1) edge[dashed] (vl2)
            (vk1) edge[dashed]  (v1) (v1) -- (vl1) (vl1) edge[dashed] (vk2);
        \draw[thick,dashdotted] plot [smooth, tension=2] coordinates {
            (v2.south) (3.0,-3.0) (vk2.south east) };
        \draw[thick,dashdotted] plot [smooth, tension=2] coordinates {
            (vk1.south) (-1.0,-2) (vj1.south west) };
        \draw[thick,dashdotted] plot [smooth, tension=2] coordinates {
            (vj2.south) (4.5,-2.5) (vl2.south east) };
    \end{tikzpicture}
\end{minipage}

\paragraph{Case 7.} Suppose that $k$ is odd.

\paragraph{Case 7a.} If $v_k=v_{1}$ and $v_{k+1}=v_{\lambda-1}$, then let

\begin{minipage}[c]{.5\textwidth}
\begin{align*}
	s_1&=v_{k+1}v_{k+2}\ldots v_{\lambda-2}v_{\lambda-1},\\
	s_2&=\left\{\begin{array}{ll}
        v_0v_1v_2\ldots v_{k-1}v_k v_{k+1}v_{\lambda}v_0 &\text{ if }k>1,\\
        v_0v_1v_{2}v_{\lambda}v_0 &\text{ if }k=1.
        \end{array}\right.
\end{align*}
\end{minipage}%
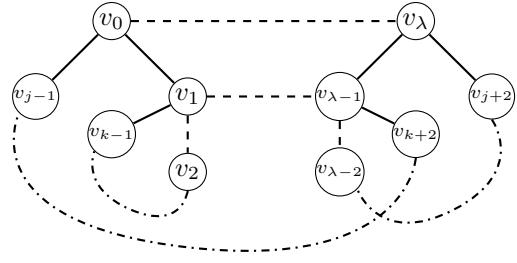
\begin{minipage}[c]{.5\textwidth}
    \centering
    \begin{tikzpicture}
        \node[draw, circle, inner sep=0pt, minimum size=14pt] (v0) at (0,0) {$v_0$};
        \node[draw, circle, inner sep=0pt, minimum size=14pt] (v1) at (1,-1) {$v_1$};
        \node[draw, circle, inner sep=0pt, minimum size=14pt] (v2) at (1,-2) {$v_2$};
        \node[draw, circle, inner sep=0pt, minimum size=14pt] (vj1) at (-1,-1) {\scriptsize $v_{j-1}$};
        \node[draw, circle, inner sep=0pt, minimum size=14pt] (vj2) at (5,-1) {\scriptsize $v_{j+2}$};
        \node[draw, circle, inner sep=0pt, minimum size=14pt] (vl0) at (4,0) {$v_\lambda$};
        \node[draw, circle, inner sep=0pt, minimum size=14pt] (vl1) at (3,-1) {\scriptsize $v_{\lambda-1}$};
        \node[draw, circle, inner sep=0pt, minimum size=14pt] (vl2) at (3,-2) {\scriptsize $v_{\lambda-2}$};

        \node[draw, circle, inner sep=0pt, minimum size=14pt] (vk1) at (0,-1.5)
            {\scriptsize $v_{k-1}$};
        \node[draw, circle, inner sep=0pt, minimum size=14pt] (vk2) at (4,-1.5)
            {\scriptsize $v_{k+2}$};

        \draw[thick] (v0)--(v1) (v0) -- (vj1) (vl0)--(vl1) (vl0) -- (vj2) (v0)
            edge[dashed] (vl0) (v1) edge[dashed] (v2) (vl1) edge[dashed] (vl2)
            (vk1) -- (v1) (v1) edge[dashed]  (vl1) (vl1) -- (vk2);
        \draw[thick,dashdotted] plot [smooth, tension=2] coordinates {
            (v2.south) (0.25,-2.5) (vk1.south west) };
        \draw[thick,dashdotted] plot [smooth, tension=2] coordinates {
            (vk2.south) (1,-3) (vj1.south west) };
        \draw[thick,dashdotted] plot [smooth, tension=2] coordinates {
            (vj2.south) (4.5,-2.5) (vl2.south east) };
    \end{tikzpicture}
\end{minipage}
\paragraph{Case 7b.} If $v_{k+1}=v_{1}$ and $v_{k}=v_{\lambda-1}$, then let

\begin{minipage}[c]{.5\textwidth}
\begin{align*}
	s_1&=v_{0}v_{1}v_2\ldots v_{k-2}v_{k-1}v_{\lambda-1}v_{\lambda-2}\ldots
    v_{j+2}v_{j+1}v_j,\\
    s_2&=v_{k}v_{k+1}v_{k+2}\ldots v_{j-1}v_{j}v_{j+1}v_{\lambda-1}.
\end{align*}
\end{minipage}%
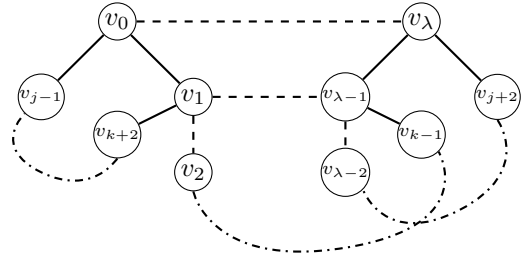
\begin{minipage}[c]{.5\textwidth}
    \centering
    \begin{tikzpicture}
        \node[draw, circle, inner sep=0pt, minimum size=14pt] (v0) at (0,0) {$v_0$};
        \node[draw, circle, inner sep=0pt, minimum size=14pt] (v1) at (1,-1) {$v_1$};
        \node[draw, circle, inner sep=0pt, minimum size=14pt] (v2) at (1,-2) {$v_2$};
        \node[draw, circle, inner sep=0pt, minimum size=14pt] (vj1) at (-1,-1) {\scriptsize $v_{j-1}$};
        \node[draw, circle, inner sep=0pt, minimum size=14pt] (vj2) at (5,-1) {\scriptsize $v_{j+2}$};
        \node[draw, circle, inner sep=0pt, minimum size=14pt] (vl0) at (4,0) {$v_\lambda$};
        \node[draw, circle, inner sep=0pt, minimum size=14pt] (vl1) at (3,-1) {\scriptsize $v_{\lambda-1}$};
        \node[draw, circle, inner sep=0pt, minimum size=14pt] (vl2) at (3,-2) {\scriptsize $v_{\lambda-2}$};

        \node[draw, circle, inner sep=0pt, minimum size=14pt] (vk1) at (0,-1.5)
            {\scriptsize $v_{k+2}$};
        \node[draw, circle, inner sep=0pt, minimum size=14pt] (vk2) at (4,-1.5)
            {\scriptsize $v_{k-1}$};

        \draw[thick] (v0)--(v1) (v0) -- (vj1) (vl0)--(vl1) (vl0) -- (vj2) (v0)
            edge[dashed] (vl0) (v1) edge[dashed] (v2) (vl1) edge[dashed] (vl2)
            (vk1) -- (v1) (v1) edge[dashed]  (vl1) (vl1) -- (vk2);
        \draw[thick,dashdotted] plot [smooth, tension=2] coordinates {
            (v2.south) (3.0,-3.0) (vk2.south east) };
        \draw[thick,dashdotted] plot [smooth, tension=2] coordinates {
            (vk1.south) (-1.0,-2) (vj1.south west) };
        \draw[thick,dashdotted] plot [smooth, tension=2] coordinates {
            (vj2.south) (4.5,-2.5) (vl2.south east) };
    \end{tikzpicture}
\end{minipage}

\paragraph{Case 8.} The remaining case is when $k>j$. By taking the reverse order
$v'_i=v_{\lambda-i}$ for $i=0,\ldots,\lambda$, we have
$\lambda-k-1<\lambda-j-1$, so one of the previous subcases of Case 6 or Case 7 applies to
$s'=v'_0v'_1\ldots v'_{\lambda-1}v'_{\lambda}$. Clearly, the relevant properties
of $s_i$ are preserved by reversing the order of the indices.

\nocite{*}
\renewcommand*{\bibfont}{\small}
\printbibliography{}

\end{document}